%% file: main.tex
\documentclass[11pt,letterpaper]{amsart}

\usepackage{amssymb,amscd,amsmath,amsthm}
\usepackage[T1]{fontenc}
\usepackage{lmodern}
\usepackage{microtype}
\usepackage{graphicx}
\usepackage{mathtools}
\usepackage{hyperref}
\usepackage[shortlabels]{enumitem}
\usepackage{thm-restate}
\usepackage{mathrsfs}
\usepackage[british]{babel}
\usepackage{tikz-cd}
\usepackage{float}
\usepackage[all,cmtip]{xy}

\newtheorem{thm}{Theorem}[section]
\newtheorem{cor}[thm]{Corollary}
\newtheorem{lem}[thm]{Lemma}
\newtheorem{prop}[thm]{Proposition}
\newtheorem*{thm*}{Theorem}
\newtheorem*{prop*}{Proposition}

\newcounter{theoremalph}

\newtheorem{thmAlph}[theoremalph]{Theorem}
\newtheorem{corAlph}[theoremalph]{Corollary}

\theoremstyle{definition}
\newtheorem{defn}[thm]{Definition}
\newtheorem*{defn*}{Definition}
\theoremstyle{remark}

\newtheorem{rem}[thm]{Remark}
\newtheorem{exmp}[thm]{Example}
\newcounter{theoremalphex}

\newtheorem{exmpAlph}[theoremalphex]{Example}
\newtheorem{probAlph}[theoremalphex]{Problem}

\theoremstyle{definition}

\makeatletter
\@addtoreset{proofpart}{theorem}
\makeatother

\newtheoremstyle{custom}{}{}{\itshape}{}{\bfseries}{.}{.5em}{#1 \thmnote{#3}}
\theoremstyle{custom}

\DeclareMathOperator{\Ad}{Ad}
\DeclareMathOperator{\Aut}{Aut}
\DeclareMathOperator{\cd}{cd}
\DeclareMathOperator{\Comm}{Comm}

\DeclareMathOperator{\id}{Id}
\DeclareMathOperator{\im}{Im}

\DeclareMathOperator{\Isom}{Isom}

\DeclareMathOperator{\CAT}{CAT}

\DeclareMathOperator{\Sym}{Sym}
\DeclareMathOperator{\MCG}{MCG}
\DeclareMathOperator{\QI}{QI}

\DeclareMathOperator{\Out}{Out}
\DeclareMathOperator{\Inn}{Inn}

\DeclareMathOperator{\GL}{GL}
\DeclareMathOperator{\SL}{SL}
\DeclareMathOperator{\Nil}{Nil}
\DeclareMathOperator{\Sol}{Sol}
\DeclareMathOperator{\PSL}{PSL}
\DeclareMathOperator{\qa}{{\overset{\scriptscriptstyle\text{q.a.}}{\curvearrowright}}}

\newfont{\cusfont}{alnorm}
\newcommand{\alnorm}{\text{\;\cusfont Q\; }}

\newcommand{\bbE}{\mathbb{E}}
\newcommand{\bbC}{\mathbb{C}}

\newcommand{\bbH}{\mathbb{H}}

\newcommand{\bbQ}{\mathbb{Q}}
\newcommand{\bbR}{\mathbb{R}}
\newcommand{\bbZ}{\mathbb{Z}}

\newcommand{\cB}{\mathcal{B}}

\newcommand{\cG}{\mathcal{G}}

\allowdisplaybreaks

\title{Model geometries of finitely generated groups}
\author{Alex Margolis}
\address{Alex Margolis, Department of Mathematics, Vanderbilt University, 1326 Stevenson Center, Nashville, TN 37240, USA}
\email{alexander.margolis@vanderbilt.edu}
\thanks{}

\begin{document}
\begin{abstract}
We study model geometries of finitely generated groups. If  a  finitely generated group does not contain a  non-trivial    finite rank free abelian commensurated subgroup, we show any model  geometry  is dominated by either  a  symmetric space of non-compact type,  an infinite locally finite vertex-transitive graph, or a product of such spaces. We also prove that a finitely generated  group possesses a model geometry not dominated by a locally finite graph if and only if it contains either a  commensurated finite rank free abelian subgroup, or a uniformly commensurated subgroup that is a uniform lattice in a semisimple Lie group. This characterises  finitely generated groups that embed as uniform lattices in locally compact groups that are not compact-by-(totally disconnected). We show the only such groups of cohomological two  are surface groups and generalised Baumslag--Solitar groups, and we obtain an analogous characterisation for groups of cohomological dimension three.
\end{abstract}
\maketitle
\section{Introduction}
Geometric group theory  is the study of groups via their isometric actions on metric spaces. If $\Gamma$ is a finitely generated group and $X$ is a proper quasi-geodesic metric space on which $\Gamma$ acts geometrically, we say that $X$ is a \emph{model geometry} of $\Gamma$. The  Milnor--Schwarz lemma, sometimes called the fundamental lemma of geometric group theory,  says that all model geometries of $\Gamma$ are quasi-isometric to one another. 
Every finitely generated group has a model geometry that is a locally finite vertex-transitive graph, namely its Cayley graph with respect to any finite generating set. In this article we investigate  when locally finite  vertex-transitive graphs are essentially the only  model geometries of a fixed finitely generated group.   To phrase this question precisely, we introduce the notion of domination of metric spaces. 

A key  idea in geometry, going back to at least Klein's \emph{Erlangen program}, is that the isometry group $\Isom(X)$  of a metric space $X$ better captures the salient geometric features of $X$ than  the actual  metric.   This viewpoint was crucial in  Thurston's definition of his eight model geometries, in which   infinitely many non-isometric Riemannian manifolds may  correspond to the same model geometry; see the discussion in \cite[\S 3.8]{thurston1997Threedimensional}.  For example, rescaling the metric on $\bbH^3$ by any non-zero constant does not change its isometry group and more generally, does not alter the synthetic  geometry  of hyperbolic 3-space.  

 This idea is particularly natural when studying  isometric group actions on a metric space $X$, which are simply representations of a group into $\Isom(X)$.
If $X$ is a proper metric space, then its isometry group $\Isom(X)$, endowed with the compact-open topology, is a locally compact   topological group. If $X$ is a model geometry of $\Gamma$, then modulo a finite normal subgroup, $\Gamma$ is a uniform lattice in $\Isom(X)$. 

  A homomorphism $\phi:G\to H$ between locally compact groups is said to be \emph{copci} if it is \emph{\textbf{co}ntinuous and \textbf{p}roper with  \textbf{c}ocompact \textbf{i}mage}. Named by  Cornulier \cite{cornulier2015commability}, a copci homomorphism is  an isomorphism up to compact error, and is a  generalisation of  the class of homomorphisms  between discrete groups with finite kernel and finite index image.
\begin{defn*}\label{defn:dom}
Given proper quasi-geodesic metric spaces $X$ and  $Y$, we say that \emph{$X$ is dominated by $Y$} if there is a copci homomorphism $\Isom(X)\rightarrow \Isom(Y)$.
\end{defn*}
If $X$ is dominated by $Y$ and  $X$ is a model geometry of $\Gamma$, then $Y$  is also a model geometry of $\Gamma$. In particular, $X$ and $Y$ are quasi-isometric. More specifically, $\Isom(X)$ acts isometrically on $Y$ and there is a quasi-isometry $f:X\to Y$  that is coarsely $\Isom(X)$-equivariant.    Before stating our results, we  give an illustrative  concrete example of a metric space dominating another metric space.
\begin{exmpAlph}\label{exmp:dom}
		
	Let $Z$ be the \emph{torus with two antennae} pictured in Figure \ref{fig:torus_antennae}.   Formally, $Z$ is the  length space obtained as the wedge of a flat torus with two unit intervals, attached along their endpoints and equipped with the induced path metric. The universal cover  $\widetilde{Z}$ is a model geometry of $\pi_1(Z)\cong \bbZ^2$  consisting of a copy of the Euclidean plane $\bbE^2$  with a pair of unit length  antennae attached to every point of a lattice $A\subseteq \bbE^2$.  
		\begin{figure}[b]
		\centering
		\def\svgwidth{\columnwidth}
		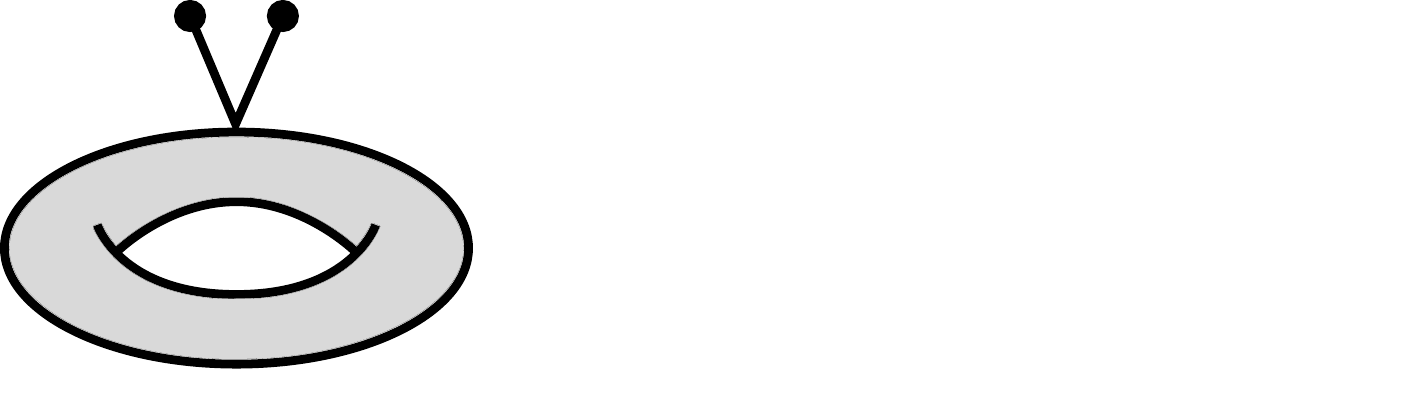
		\caption{ The universal cover of a flat torus with two antennae is dominated by the Euclidean plane.}
		\label{fig:torus_antennae}
	\end{figure}
	
	Each isometry of $\widetilde{Z}$ stabilises the subspace $\bbE^2$, so there is an induced map $\phi:\Isom(\widetilde{Z})\to \Isom(\bbE^2)$ that is  copci; thus $\widetilde{Z}$ is  dominated by $\bbE^2$. Non-trivial isometries of $\widetilde{Z}$  fixing $\bbE^2$ pointwise  lie in the compact normal subgroup $\ker(\phi)$, and should be thought of as unwanted noise in $\Isom(\widetilde{Z})$.  Passing from $\widetilde{Z}$ to $\bbE^2$ has the effect of forgetting this noise, whilst simultaneously gaining many isometries of $\bbE^2$ not stabilising  the lattice $A$. 
\end{exmpAlph}

The study of model geometries and lattice embeddings of discrete groups has a long history going back to work of Furstenberg and Mostow. Mosher--Sageev--Whyte studied model geometries of virtually free groups \cite{mosher2003quasi,moshersageevwhyte2002maximally}. Furman and Bader--Furman--Sauer classified  lattice envelopes  for a large class of groups, including lattices in Lie groups and $S$-arithmetic lattices \cite{furman2001mostowmargulis,baderfurmansauer2020lattice}. Dymarz studied model geometries of certain solvable groups \cite{dymarz2015envelopes}. %

We say a graph is  \emph{singular} if not all vertices have valence two. For ease of notation, we use the term \emph{locally finite graph} as shorthand for a connected singular locally finite graph,  equipped with the induced  path metric in which each edge has length one.  The aim of this article is to give a complete solution to the following problem. 
\begin{probAlph}\label{prob:fggroups}
	Characterise  the class of finitely generated groups all of whose  model geometries are  dominated by locally finite vertex-transitive graphs. 
\end{probAlph}

A proper quasi-geodesic metric space with cocompact isometry group is dominated by a locally finite graph if and only if it is dominated by a vertex-transitive locally finite graph.    A model geometry $X$  is dominated by a locally finite graph if and only if the identity component of $\Isom(X)$ is compact.  Problem \ref{prob:fggroups} is thus equivalent to  characterising the class	 of  finitely generated groups that are, up to a finite normal subgroup,  uniform lattices in some locally compact group that is not compact-by-(totally disconnected).

We first give some simple yet important examples of model geometries.
\begin{exmpAlph}\label{exmp:graph}
	Any finitely generated group has a model geometry that is a locally finite vertex-transitive graph, namely its Cayley graph  with respect to a finite generating set. 
\end{exmpAlph}
Let $S$ be a closed hyperbolic surface. The hyperbolic plane  is a model geometry of $\pi_1(S)$, and is an example of a symmetric space of non-compact type. Symmetric spaces are one of the most fundamental objects in geometry, and were classified by Cartan. We refer the reader to Helgason's book for more details \cite{helgason1978differential}.  
\begin{exmpAlph}\label{exmp:symm}
	Let $X$  be a symmetric space  of non-compact type. 
	Borel showed  $\Isom(X)$  contains a uniform lattice \cite{borelm1963compact}, hence $X$ is a model geometry of some finitely generated group.  As $\Isom(X)$ is virtually connected, $X$ is not dominated by a locally finite  graph. 
\end{exmpAlph}

Combining  Examples \ref{exmp:graph} and \ref{exmp:symm}, there exist model geometries that are products of symmetric spaces and locally finite vertex-transitive graphs:
\begin{exmpAlph}\label{exmp:prod}
	Let $\Gamma_1$ be a uniform lattice in the isometry group of a  symmetric space   $X$ of non-compact type,  and let $\Gamma_2$ be an arbitrary finitely generated group. Suppose $Y$ is a Cayley graph of $\Gamma_2$ with respect to a finite generating set.  Then  $X\times Y$ is a model geometry of $\Gamma_1\times \Gamma_2$. Since the identity component of  $\Isom(X\times Y)$  is non-compact, $X\times Y$ is not dominated by a locally finite  graph.
\end{exmpAlph}

As we will shortly see in Theorem \ref{thm:modelgeom}, for a large class of finitely generated groups all model geometries are dominated by one of the  model geometries described in Examples \ref{exmp:graph}--\ref{exmp:prod}. Before stating this result, we give an example of a more exotic model geometry not dominated by the model geometries described in Examples \ref{exmp:graph}--\ref{exmp:prod}. Two subgroups $\Lambda$ and $\Lambda'$ of $\Gamma$  are  \emph{commensurable} if the intersection $\Lambda\cap \Lambda'$ has finite index in both $\Lambda$ and $\Lambda'$. 
We say that $\Lambda$ is \emph{commensurated} or \emph{almost normal} in $\Gamma$, denoted $\Lambda\alnorm \Gamma$,  if every conjugate of $\Lambda$ is commensurable to $\Lambda$.
\begin{exmpAlph}\label{exmp:bs}
	Let $\Gamma$ be the Baumslag--Solitar group $BS(1,2)$ with presentation $\langle a,t\mid tat^{-1}=a^2\rangle$. The infinite cyclic subgroup $\langle a \rangle$ is commensurated. 
		The group $\Gamma$ has a piecewise Riemannian model geometry   $X$ that is a warped product of the regular 3-valent tree $T$ with $\bbR$; see for example Farb--Mosher \cite{farbmosher1998bs1}. Since $X$ admits a 1-parameter subgroup of translations along the  $\bbR$-direction, the identity component of $\Isom(X)$ is non-compact and therefore $X$ is not dominated by a locally finite graph.
\end{exmpAlph}

Our first result says that if we exclude groups containing commensurated finite rank free abelian subgroups such as the Baumslag--Solitar group in Example \ref{exmp:bs}, then  all model geometries are dominated by the model geometries described in Examples \ref{exmp:graph}--\ref{exmp:prod}.
\begin{thmAlph}\label{thm:modelgeom}
	Let $\Gamma$ be an infinite finitely generated group that does not contain a non-trivial finite rank free abelian commensurated subgroup. Any model geometry of $\Gamma$ is dominated by a space of the  form $X$, $Y$ or $X\times Y$, where:
	\begin{enumerate}
		\item $X$ is a  symmetric space of non-compact type;
		\item $Y$ is an infinite locally finite vertex-transitive graph.
	\end{enumerate}
\end{thmAlph}

We also have an analogue of Theorem \ref{thm:modelgeom} for lattices. The following theorem is strictly stronger than Theorem \ref{thm:modelgeom} as it applies  to all lattices, including non-uniform ones. 
\begin{thmAlph}\label{thm:lattice_prod}
	Let $\Gamma$ be an infinite finitely generated group that does not contain a non-trivial finite rank free abelian commensurated subgroup. Let $G$ be a locally compact group containing $\Gamma$ as a lattice. Then there is a continuous proper map $\phi:G\to S\times D$, with  compact kernel and finite index open image, where $S$ is a centre-free semisimple Lie group with finitely many components and no compact factors, and $D$ is a compactly generated totally disconnected locally compact group. 
\end{thmAlph}

 A subset $\Omega\subseteq \QI(X)$ of the quasi-isometry group of a space $X$ is said to be \emph{uniform} if there exist constants $K$ and $A$ such that every element of $\Omega$ can be represented by a $(K,A)$-quasi-isometry. 
\begin{defn*}
	A finitely generated commensurated subgroup $\Lambda\alnorm \Gamma$ is \emph{uniformly commensurated} if  the image of  the natural map \[\Gamma\to \Comm(\Lambda)\to \QI(\Lambda)\] induced by conjugation  is a uniform subgroup of $\QI(\Lambda)$.
\end{defn*}
Elaborating on this definition, if $\Lambda\alnorm \Gamma$ is a finitely generated commensurated subgroup, then for each $g\in \Gamma$ there is always \emph{some} quasi-isometry $f_g:\Lambda\to \Lambda$ that agrees with conjugation by $g$ on some finite index subgroup of $\Lambda$. The preceding definition says  $\Lambda$ is uniformly commensurated precisely when all the  $f_g$ can be chosen with  quasi-isometry constants independent of $g$. 
We now give an example of a  normal  hence commensurated subgroup  that is not  uniformly commensurated.
\begin{exmpAlph}
		Let $S$ be a closed hyperbolic surface and let $H\leq \MCG(S)$ be infinite. Let $\Gamma_H$ be the surface group extension fitting into the short exact sequence \[1\to \pi_1(S)\to \Gamma_H\to H\to 1,\] where the action of $H$ on $\pi_1(S)$ is given by $H\leq \MCG(S)\cong \Out(\pi_1(S))$.  Elements of $H$  induce automorphisms of  $\pi_1(S)$ with  arbitrarily large quasi-isometry constants. This follows from Thurston's \emph{quasi-isometry metric} on Teichm\"uller space \cite[\S 4.6]{thurston1997Threedimensional} or from Corollary \ref{cor:unif_comm}. Therefore,  the normal  subgroup $\pi_1(S)\vartriangleleft \Gamma_H$ is \emph{not} uniformly commensurated. 
\end{exmpAlph}

 Two groups $\Gamma_1$ and $\Gamma_2$ are \emph{virtually isomorphic} if for $i=1,2$  there exist finite index subgroup $\Gamma'_i\leq \Gamma_i$ and  finite normal subgroups $F_i\vartriangleleft \Gamma'_i$ such that  $\Gamma'_1/F_1$ and $\Gamma'_2/F_2$  are isomorphic. The following is the main result of this article and provides a complete solution to Problem \ref{prob:fggroups}. 
\begin{thmAlph}\label{thm:mainintro}
	Let $\Gamma$ be a finitely generated group. The following are equivalent:
	\begin{enumerate}
		\item $\Gamma$ has a model geometry that is not dominated by a locally finite  graph.\label{item:mainintro1}
		\item  $\Gamma$ contains a finite normal subgroup $F$ such that $\Gamma/F$ is a uniform lattice in a locally compact group whose identity component is non-compact.	\label{item:mainintro2}
		\item $\Gamma$  contains an infinite commensurated subgroup $\Lambda$ such that one of the following hold:\label{item:mainintro3}
		\begin{enumerate}
			\item $\Lambda$ is  a finite rank free abelian group; \label{mainthm:abelian}
			\item $\Lambda$ is uniformly commensurated and is virtually isomorphic to a uniform lattice in a connected centre-free semisimple Lie group  without compact factors.\label{mainthm:semisimple}
		\end{enumerate}
	\end{enumerate}  
\end{thmAlph}
The equivalence of the first two conditions of Theorem \ref{thm:mainintro} is not difficult; see Section \ref{sec:locally_compact_prelims} for details. The content of the theorem is therefore that the first two conditions  are equivalent to  the third.
We also remark that the uniformly  commensurated hypothesis in (\ref{mainthm:semisimple}) holds automatically if   $\Lambda$ is virtually isomorphic to a lattice  satisfying  Mostow rigidity; see Proposition \ref{prop:unifcomm_mostow}.

 To understand the geometric and topological structure of 3-manifolds,  Thurston defined his eight \emph{model geometries}, which are $\bbH^3$, $\bbE^3$, $\Sol$, $\bbH^2\times \bbR$,  $\widetilde{\SL(2,\bbR)}$, $\Nil$, $S^3$ and $S^2\times \bbE^1$. 
   A closed \emph{geometric 3-manifold} is a 3-manifold obtained by taking the quotient of one of the eight model geometries  by a group of isometries acting freely and cocompactly. 
The  \emph{Geometrisation Theorem}, conjectured by Thurston in 1982 and proved by Perelman in 2002, states that every closed 3-manifold  is either geometric or can be decomposed into geometric pieces \cite{thurston1982threedimensional,perelman2002entropy}.
If $M$ is a closed geometric 3-manifold with infinite fundamental group, then its fundamental group  has a model geometry not dominated by a locally finite graph.

In Theorems \ref{thm:cd1}--\ref{thm:cd3} we apply Theorem \ref{thm:mainintro} to give structural  descriptions of low dimensional groups that have model geometries not dominated by locally finite graphs. These results are an analogue of Thurston's classification of his eight model geometries from the setting of 3-manifold groups to finitely generated groups of low cohomological dimension.
The following result  is a straightforward consequence of theorems of Stallings and  Mosher--Sageev--Whyte \cite{stallings1968torsionfree,mosher2003quasi}, which we  state  for completeness and comparison with Theorems \ref{thm:cd2} and \ref{thm:cd3}.
\begin{thmAlph}\label{thm:cd1}
	Let $\Gamma$ be a finitely generated group of cohomological dimension one. Then $\Gamma$ has a model geometry that is not dominated by a locally finite vertex-transitive graph if and only if  $\Gamma$ is infinite cyclic.
\end{thmAlph}
A  \emph{surface group} is the fundamental group of a closed  surface of non-positive Euler characteristic.
  A \emph{generalised Baumslag--Solitar group of rank $n$}  is the fundamental group of a finite graph of groups in which every vertex and edge group is virtually  $\bbZ^n$, and the associated Bass--Serre tree is infinite-ended. 
  
   It was already observed in Examples \ref{exmp:symm} and  \ref{exmp:bs} that surface groups and  generalised Baumslag--Solitar groups of rank one  admit  model geometries that are not dominated by a locally finite  graph. The next theorem shows these are the only such examples amongst groups of cohomological dimension two.
\begin{thmAlph}\label{thm:cd2}
	Let $\Gamma$ be a finitely generated group of cohomological dimension two. Then $\Gamma$ has a model geometry that is not dominated by a locally finite graph if and only if either:
	\begin{enumerate}
		\item  $\Gamma$ is a surface group, hence acts geometrically on $\bbE^2$ or $\bbH^2$. \label{item:geom2}
		\item   $\Gamma$ is a generalised Baumslag--Solitar group of rank one.  \label{item:baum2}
	\end{enumerate}
\end{thmAlph}
Baumslag--Solitar groups are one of jewels in the crown of combinatorial and geometric group theory, and have received  much attention over the years \cite{baumslagsolitar1962,collinslevin1983automorphisms,kropholler1990CD2,farbmosher1998bs1,farbmosher1999bs2,whyte2001baumslag,mosher2003quasi,levitt2007automorphism}.
Theorem \ref{thm:cd2}  further  highlights their centrality and importance.

We obtain an analogous  classification for groups of cohomological dimension three.
\begin{thmAlph}\label{thm:cd3}
	Let $\Gamma$ be a finitely generated group of cohomological dimension three. Then $\Gamma$ has a model geometry that is not dominated by a locally finite graph if and only if either:
	\begin{enumerate}
			\item $\Gamma$ acts geometrically on  $\bbH^3$, $\bbE^3$ or $\Sol$. \label{item:3mfld}
			\item   $\Gamma$ acts geometrically on $\bbH^2\times T$ for some locally finite infinite-ended tree $T$.
			\item  $\Gamma$ is a  generalised Baumslag--Solitar group of rank two. 
			\item  $ \Gamma$ contains an infinite cyclic commensurated subgroup. \label{item:zalnorm}
		\end{enumerate} 
\end{thmAlph}
We comment on the four alternatives in Theorem \ref{thm:cd3}.
	The spaces $\bbH^3$, $\bbE^3$ and $\Sol$ are three of Thurston's eight model geometries. If $\Gamma$ acts geometrically on any of  the  three remaining  aspherical Thurston model geometries, namely    $\bbH^2\times \bbR$,  $\widetilde{\SL(2,\bbR)}$ or $\Nil$, then $\Gamma$ contains an infinite cyclic normal subgroup, so is encompassed by alternative (\ref{item:zalnorm}) of Theorem \ref{thm:cd3}. Fundamental groups of closed 3-manifolds modelled on the spherical-type model geometries $S^3$ and $S^2\times \bbE^1$ do not have cohomological dimension three.
	
	Groups acting geometrically on $\bbH^2\times T$ include arithmetic lattices in $\Isom(\bbH^2)\times \PSL(\bbQ_p)$, as well as more exotic non-residually finite examples constructed  by Hughes--Valiunas \cite{hughes2022commensurating}. Generalised Baumslag--Solitar groups of rank two were classified up to quasi-isometry by Mosher--Sageev--Whyte, Farb--Mosher and Whyte \cite{farbmosher2000abelianbycyclic,whyte2001baumslag,mosher2003quasi,whyte2010coarse}.  Leary--Minasyan recently provided the first examples of  $\CAT(0)$ groups that are not biautomatic; their groups are generalised Baumslag--Solitar group of rank two, and act geometrically on $\bbE^2\times T$.  Groups in (\ref{item:zalnorm}) comprise a large class of groups  including  groups of the form $\bbZ\times \Lambda$ where $\Lambda$ is any finitely presented group of cohomological dimension two.

	\subsection*{Two types of rigidity}
In order to explain the phenomena underlying the results of this article, we discuss two different types of  rigidity that lattices in connected Lie groups may possess. We motivate the discussion with  three examples of lattices.
\begin{enumerate}
	\item $\bbZ^n$ for $n\geq 1$, which is a uniform lattice in $\bbR^n$.
	\item $\pi_1(S)$ for $S$ a closed hyperbolic surface, which is a uniform lattice in $\PSL(2,\bbR)$.
	\item $\pi_1(M)$ for $M$ a closed hyperbolic 3-manifold, which is a uniform lattice in $\PSL(2,\bbC)$.
\end{enumerate}

The hyperbolic 3-manifold group $\pi_1(M)$ is undoubtedly the most rigid, satisfying both types of rigidity.
Firstly, every automorphism of $\pi_1(M)$ induces an automorphism of $\PSL(2,\bbC)$. We call this  \emph{lattice rigidity}.
Secondly, the locally compact group $\PSL(2,\bbC)$ containing  $\pi_1(M)$ as a lattice is itself rigid, since   $\Out(\PSL(2,\bbC))$ is finite.  We call this \emph{envelope rigidity}. 

The free abelian group $\bbZ^n$ satisfies lattice rigidity, since every automorphism of $\bbZ^n$ induces an automorphism of $\bbR^n$. However, $\Out(\bbR^n)=\Aut(\bbR^n)=\GL_n(\bbR)$ is infinite, so $\bbZ^n$ does not satisfy envelope rigidity. In contrast, $\pi_1(S)$ does not satisfy lattice rigidity, since infinite order mapping classes in $\MCG(S)\cong \Out(\pi_1(S))$ do not induce automorphisms of $\PSL(2,\bbR)$. However, $\PSL(2,\bbR)$ has finite outer automorphism group, so $\pi_1(S)$ does satisfy envelope rigidity.

Envelope rigidity is the phenomenon responsible for the product structure of model geometries in Theorem \ref{thm:modelgeom}. The fact that $\bbZ^n$ is not  envelope rigid ensures commensurated abelian subgroups may give rise to model geometries with a warped product structure such as the Baumslag--Solitar group $BS(1,2)$ in Example \ref{exmp:bs}. Lattice rigidity of free abelian groups ensures that every commensurated $\bbZ^n$ subgroup gives rise to some model geometry not dominated by a locally finite graph as in Theorem \ref{thm:mainintro}. 
In contrast, the fact that  $\pi_1(S)$ is envelope rigid but not lattice rigid means that whilst groups acting geometrically  on  $\bbH^2\times T$ appear  in the statement of Theorem \ref{thm:cd3},  surface-by-free groups of the form considered in Example \ref{exmp:symm} do not.

	\subsection*{Connection to work of Bader--Furman--Sauer} In  \cite{baderfurmansauer2020lattice}, Bader--Furman--Sauer introduce the  conditions (Irr), (CAF) and (NbC)  and  study lattice envelopes of   countable groups that satisfy these conditions.  When these conditions are satisfied, they obtain results that are stronger and more complete than the results of this article.  The results of this article complement \cite{baderfurmansauer2020lattice} by providing information on finitely generated groups in which at least one of the  conditions (Irr), (CAF) and (NbC) does not hold.  Both the statement and  proof of Proposition \ref{prop:lattice_structure} of this article  are similar to and inspired by  Proposition 5.3 of \cite{baderfurmansauer2020lattice}. However,  we  replace the (CAF) condition, which says the group does not contain an infinite amenable  commensurated subgroup, with the condition that the group does not contain an infinite finite rank abelian commensurated subgroup.

	\subsection*{Relation to discretisability} In \cite{margolis2022discretisable}, the author introduced  discretisable spaces. These are spaces for which every cobounded quasi-action  can be quasi-conjugated to an isometric action on a locally finite graph. If the space in question is  a finitely generated group $\Gamma$, then discretisability  is a coarse analogue of the property that every model geometry of $\Gamma$ is dominated by a locally finite graph. Indeed, it is shown in 
	\cite[Corollary 4.15]{margolis2022discretisable} that if a finitely generated group $\Gamma$ is a lattice in a locally compact group that is not compact-by-(totally disconnected), then $\Gamma$ is not discretisable. Combined with Theorem \ref{thm:mainintro}, this proves:
	\begin{corAlph}\label{cor:discretisable}
		Let $\Gamma$ be a finitely generated group that contains an infinite commensurated subgroup that is either:
		\begin{enumerate}
			\item a finite rank free abelian group
			\item a uniformly commensurated group virtually isomorphic to a uniform lattice in a centre-free semisimple Lie group with finitely many components.
		\end{enumerate}Then $\Gamma$ is not discretisable. 
	\end{corAlph}
	All examples of non-discretisable groups that the author is aware of arise from Corollary \ref{cor:discretisable}. This corollary  was one of the motivations of the author for writing this article, giving sufficient algebraic conditions for a group to not be discretisable. It will be used in upcoming work of the author to further investigate which groups are and are not  discretisable.
	
	\subsection*{Outline of article and discussion of proofs}
In Section \ref{sec:locally_compact_prelims} we gather  preliminary results concerning locally compact groups,  uniform lattices, model geometries and copci homomorphisms. In Section \ref{sec:engulf} we discuss engulfing groups of  commensurated subgroups.  To illustrate the idea, if $\bbZ$ is a commensurated subgroup of $\Gamma$, there is a canonical map $\Gamma\to\Aut(\bbR)$ compatible with the action of $\Gamma$ by conjugation on subgroups of $\bbZ$ of sufficiently large finite index. In such a  situation, we say $\bbR$ is an engulfing group.

Section \ref{sec:structure} is devoted to the proof of Proposition \ref{prop:lattice_structure}, from which we deduce Theorems \ref{thm:modelgeom}, \ref{thm:lattice_prod} and one direction of Theorem \ref{thm:mainintro}. The proof of Proposition \ref{prop:lattice_structure} is a combination of several  results, the most important of which is a theorem of Gleason--Yamabe solving  Hilbert's fifth problem. Coupled with van Dantzig's theorem, this essentially reduces Proposition \ref{prop:lattice_structure} to a problem about lattices in Lie groups. Applying tools from the structure theory of Lie groups  and a theorem of Mostow regarding  lattice-hereditary subgroups of Lie groups allows us to deduce Proposition \ref{prop:lattice_structure}. 

In Section \ref{sec:building_group} we further develop the notion of an engulfing group. The main technical result is Theorem \ref{thm:build_lattice_embedding}. Informally, this theorem can be thought of as a way of approximating a commensurated subgroup with a normal subgroup by embedding the ambient group into a locally compact group. More precisely, Theorem \ref{thm:build_lattice_embedding} takes as input a  finitely generated group $\Gamma$ containing a commensurated subgroup $\Lambda$ with engulfing group $E$, and embeds $\Gamma$ as a uniform lattice in a locally compact group $G$ containing $E$ is a closed normal subgroup. The quotient $G/E$ is the Schlichting completion  of $\Gamma$ with respect to $\Lambda$. Theorem \ref{thm:build_lattice_embedding}  builds a group extension  using a continuous generalised cocycle that is determined by  $\Gamma$ and $\Lambda$. Theorem \ref{thm:mainintro} follows readily from Theorem \ref{thm:build_lattice_embedding}. We anticipate Theorem \ref{thm:build_lattice_embedding} to be a useful stand-alone  result for investigating algebraic and geometric properties of  commensurated subgroups.

We conclude with proofs of  Theorems \ref{thm:cd2}  and \ref{thm:cd3} in Section \ref{sec:lowdim}. The main ingredient is  work of the author concerning commensurated subgroups of cohomological codimension one \cite{margolis2019codim1}, generalising results of Kropholler \cite{kropholler1990CD2,kropholler2006spectral}. We also highlight Proposition \ref{prop:extension_action}, a result of independent interest that uses work  of Grunewald--Platonov \cite{grunewaldplatonov2004new} to extend a geometric action of a finite index torsion-free subgroup on a model geometry in the sense of Thurston to an  action of the  ambient group. This obviates the need to pass to finite index subgroups in the conclusions of Theorems \ref{thm:cd2}  and \ref{thm:cd3}.
	
\subsection*{Acknowledgements}
The author would like to thank Sam Shepherd for helpful feedback and comments on a draft of this article.

\section{Locally compact groups and  metric spaces}\label{sec:locally_compact_prelims}
Topological groups are always  assumed to be Hausdorff. A topological group is \emph{locally compact} if every point admits a compact neighbourhood, or equivalently, every point admits a fundamental system of compact neighbourhoods. A topological group is \emph{compactly generated} if it contains a compact generating set. We refer the reader to \cite{hewittross1979abstract} for general background on topological groups, and to \cite{cornulierdlH2016metric} for background on the coarse geometry of locally compact groups.

All Lie groups are assumed to be real Lie groups. As a continuous map between Lie groups is automatically smooth, a topological group that is   topologically isomorphic to a Lie group can be endowed with a unique smooth structure making it a Lie group. We thus say a topological group $G$, not a priori  equipped with any  smooth manifold structure, \emph{is} a Lie group if it is topologically isomorphic to a Lie group.
The following result describes the  structure of Lie groups with finitely many components:
\begin{thm}[see e.g.\ {\cite[\S VII]{borel98semisimple}}]\label{thm:maxcompact}
	Let $G$ be a Lie group with finitely many components. Then  $G$ contains  a maximal compact subgroup $K\leq G$, unique up to conjugation. Moreover, $G$ is diffeomorphic to $K\times \bbR^n$ and so  the homogeneous space $G/K$ is diffeomorphic to $\bbR^n$.
\end{thm}

 If $G$ is a topological group, let  $G^\circ$ denote the connected component of the identity. A topological space is \emph{totally disconnected} if every connected component is a singleton. The following elementary proposition says that  to some extent,  the structure of locally compact groups reduces to the structure of connected and of totally disconnected locally compact groups. 

\begin{prop}\label{prop:idcomp}
	If $G$ is a locally compact group, then
	\begin{itemize}
		\item $G^\circ$ is a closed normal subgroup of $G$;
		\item The quotient group $G/G^\circ$ is a totally disconnected locally compact group.
	\end{itemize}
In particular,  $G$ is compact-by-(totally disconnected) if and only if $G^\circ$ is compact.
\end{prop}

Let  $(X,d_X)$ and $(Y,d_Y)$ be metric spaces. For constants $K\geq 1$ and $A\geq 0$,   a map $f:X\to Y$ is a \emph{$(K,A)$-quasi-isometry} if:
\begin{enumerate}
	\item for all $x,x'\in X$\[\frac{1}{K}d_X(x,x')-A\leq d_Y(f(x),f(x'))\leq Kd_X(x,x')+A;\] 
	\item for all $y\in Y$ there exists an $x\in X$ such that $d_Y(f(x),y)\leq A$. 
\end{enumerate}
We say $f$ is a \emph{quasi-isometry} if it is a $(K,A)$-quasi-isometry for some $K\geq 1$ and $A\geq 0$. Two quasi-isometries $f,g:X\to Y$ are \emph{$A$-close} if $d(f(x),g(x))\leq A$ for all $x\in X$, and are \emph{close} if they are $A$-close for some $A$.

A metric space is \emph{proper} if closed balls are compact, and \emph{quasi-geodesic} if it is quasi-isometric to a geodesic metric space. A metric space $X$ is said to be \emph{cocompact} if the natural action of $\Isom(X)$ on $X$ is cocompact. If $X$ is a metric space, then its isometry group $\Isom(X)$ inherits the structure of a topological group endowed with the compact-open topology. 
Moreover, if $X$ is cocompact, proper  and quasi-geodesic, then $\Isom(X)$ is second countable, locally compact and compactly generated \cite[Lemma 5.B.4 and Proposition 5.B.10]{cornulierdlH2016metric}.

We now define what it means for a locally compact group to act geometrically on a proper metric space. The following definition   is stronger than that considered in \cite{cornulierdlH2016metric}, where  $X$ need not be proper and $\rho$ need not be continuous.
\begin{defn}
	Let $G$ be a locally compact topological group and $X$ be a proper	 metric space. An isometric action  $\rho:G\rightarrow \Isom(X)$ is said to be:
	\begin{itemize}
		\item \emph{continuous} if $\rho$ is continuous;
		\item \emph{proper} if for every compact set $K\subseteq X$, the set \[\{g\in G\mid \rho(g)K\cap K\neq \emptyset\}\] has compact closure;
		\item \emph{cocompact} if there is a compact set $K\subseteq X$ such that $X=GK$.
	\end{itemize}
	An  action of $G$ on $X$ is said to be \emph{geometric} if it is isometric, continuous, proper and cocompact.
\end{defn}

\begin{rem}\label{rem:isom_ct}
	 Continuity of $\rho:G\to \Isom(X)$ is equivalent to continuity of the map $\alpha:G\times X\to X$ given by $\alpha(g,x)=\rho(g)(x)$.
\end{rem}

 A continuous  map between topological spaces is said to be \emph{proper} if the preimage of a compact set is compact.  Recall a homomorphism $G\to H$ is said to be \emph{copci} if it is continuous and proper with cocompact image.  A copci homomorphism $\phi:G\to H$ between locally compact groups factors as \[G\to G/\ker(\phi)\xrightarrow{\widehat \phi} \phi(G)\to H,\]  where $\phi(G)$ is a closed subgroup of $H$ equipped with the subspace topology and $\widehat{\phi}$ is a topological isomorphism \cite[\S 1.10]{bourbaki98gentop1-4}.
  The following lemma shows an action $\rho:G\to \Isom(X)$ is geometric if and only if $\rho$ is copci. 
\begin{lem}\label{lem:copci_action}
	Let $X$ be a proper quasi-geodesic cocompact metric space. 
	\begin{enumerate}
		\item The action of $\Isom(X)$ on $X$ is geometric. \label{item:isom_action}
		\item Let $G$ be a locally compact group. An action $\rho:G\to\Isom(X)$ is geometric if and only if it is copci.  \label{item:copci_action}
	\end{enumerate}
\end{lem}
\noindent The first part of Lemma \ref{lem:copci_action}  was shown in \cite[Lemma 5.B.4 and Proposition 5.B.10]{cornulierdlH2016metric}. The second part   was noted by Cornulier \cite[p. 117]{cornulier2015commability} and readily follows from the definitions.

We recall from the introduction that if $X$ and $Y$ are metric spaces, then  $X$ is \emph{dominated by} $Y$ if there is a   copci homomorphism $\phi:\Isom(X)\rightarrow\Isom(Y)$. Since the composition of two copci homomorphisms is copci, Lemma \ref{lem:copci_action} can be used to deduce the following  properties of domination.
\begin{lem}\label{lem:dom}
	Let $X$ be a cocompact proper quasi-geodesic metric space. 
	\begin{enumerate}
		\item If $X$ is dominated by $Y$ and  $G$ acts geometrically on $X$, then $G$ acts geometrically on $Y$.\label{dom:geomact}
		\item Suppose a locally compact group $G$ acts geometrically on a proper quasi-geodesic metric space $Y$. If  there is a  copci homomorphism  $\phi:\Isom(X)\to G$, then  $X$ is dominated by $Y$.\label{dom:dominate}
	\end{enumerate}
\end{lem}
A locally compact compactly generated group can be considered as a metric space by equipping it with the word metric with respect to a compact generating set. This metric is well-defined up to equivariant quasi-isometry. More generally, such a group can be equipped with any \emph{geodesically adapted} metric in the sense of \cite{cornulierdlH2016metric}.
\begin{lem}[{\cite[Theorem 4.C.5 and Remark 4.C.13]{cornulierdlH2016metric}}]\label{lem:quasi-isom_copci}
	Let $G$ and $H$ be locally compact compactly generated groups and let $X$ be a cocompact proper quasi-geodesic metric space.
	\begin{enumerate}
		\item A copci homomorphism $\phi:G\to H$ is a quasi-isometry.
		\item For any $x_0\in X$, the orbit map $\Isom(X)\to X$ given by $\phi\mapsto \phi(x_0)$ is a quasi-isometry.
	\end{enumerate}
\end{lem}

A (left-invariant) \emph{Haar measure} $\mu$ on a locally compact group $G$ is a Radon measure on the Borel $\sigma$-algebra $\cB$ of $G$ such that  $\mu(A)=\mu(gA)$ for all $g\in G$ and $A\in \cB$. On each locally compact group the Haar measure exists and is unique up to rescaling.

\begin{defn}
	Let $G$ be a locally compact group and let $\mu$ be a Haar measure on $G$. A subgroup $\Gamma\leq G$ is a \emph{lattice} if it is discrete and  there exists a set $K\subseteq G$ such that $\Gamma K=G$ and $\mu(K)<\infty$. If $K$ can be chosen to be compact, we say $\Gamma$ is a \emph{uniform} or cocompact lattice.
	If $\Lambda$ is discrete, a homomorphism $\rho:\Lambda\to G$ is a \emph{virtual (uniform) lattice embedding} if $\ker(\rho)$ is finite and $\im(\rho)$ is a (uniform) lattice.
\end{defn}
\begin{rem}
	The terminology virtual lattice embedding was chosen to be consistent with the notion of a virtual isomorphism.
\end{rem}

\begin{prop}\label{prop:discrete_copci}
	Let $\Gamma$ be a discrete group, $G$ be a locally compact group and  $\rho:\Gamma\to G$  be a homomorphism. The following are equivalent:
	\begin{enumerate}
		\item $\rho$ is a virtual uniform lattice embedding;\label{item:discrete_copci_lattice}
		\item $\rho$ is copci.\label{item:discrete_copci_copci}
	\end{enumerate}
\end{prop}
\begin{proof}
	Note that continuity of $\rho$ is automatic since $\Gamma$ is discrete.
	
	(\ref{item:discrete_copci_lattice})$\implies$(\ref{item:discrete_copci_copci}):  As $\rho(\Gamma)$ is a uniform lattice, $\rho$ has cocompact image.  Moreover, since $\rho$ has finite kernel and $\rho(\Gamma)$ is discrete, $\rho^{-1}(K)$ is finite for every compact $K\subseteq G$, hence $\rho$ is proper.
	
	(\ref{item:discrete_copci_copci})$\implies$(\ref{item:discrete_copci_lattice}):  Since $\rho$ is proper and $\Gamma$ is discrete, $\ker(\rho)$ is finite and $\rho(\Gamma)\cap K$ is finite for every compact $K\subseteq G$. Thus   $\rho(\Gamma)$ is discrete. Since $\rho(\Gamma)$ is cocompact, $\rho$ is a virtual uniform lattice embedding.
\end{proof}

We now restrict our attention to  groups that act geometrically on two important classes of metric spaces: symmetric spaces and locally finite graphs.
A \emph{symmetric space} is a connected Riemannian manifold $X$ such that for each $p\in X$, there is an isometry $\phi_p:X\to X$ such that $\phi_p(p)=p$ and  $(d\phi_p)_p=-\id$. Such spaces are sometimes called global symmetric spaces.  Symmetric spaces were classified by Cartan.  

A symmetric space is said to be \emph{of non-compact type} if it has non-positive sectional curvature and does not split isometrically  as a Cartesian product of the form $X=\bbE^n\times X'$ for some $n>0$.  Every symmetric space of non-compact type admits a canonical \emph{de Rham decomposition} $X=X_1\times \dots\times X_n$, where  $X_1,\dots,X_n$ are  symmetric spaces of non-compact type that are \emph{irreducible}, i.e.\ do not decompose as a non-trivial direct product.   

A  Lie group $G$ is  \emph{semisimple with no compact factors} if its Lie algebra is semisimple and $G$ contains no connected  non-trivial compact normal subgroup. Note  $G$ is semisimple with no compact factors if and only if its identity component $G^\circ$ is. 
There is a correspondence between symmetric spaces of non-compact type and semisimple Lie groups. 
\begin{prop}[{{\cite[Chapters II, IV and VI]{helgason1978differential}}}]\label{prop:manifold_gp_action}
	 Let $X$ be a symmetric space of non-compact type. Then $\Isom(X)$ has a Lie group structure acting smoothly, compatible with the compact-open topology. Moreover, $\Isom(X)$  is a centre-free semisimple Lie group with no compact factors and finitely many  components and $\Isom(X)^\circ$ acts transitively on $X$.
	 
	 Conversely, let $G$ be a semisimple Lie group with finite centre and finitely many  components. Let $K\leq G$ be a maximal compact subgroup. Then $G/K$, equipped with any $G$-invariant  Riemannian metric, is a symmetric space of non-compact type. In particular, $G$ acts geometrically on a symmetric space of non-compact type. 
\end{prop}

For the purposes of this article, a \emph{graph} is a combinatorial 1-dimensional cell complex, where \emph{vertices} are 0-cells and \emph{edges} are 1-cells.   A connected graph $X$ will be considered as a metric space by endowing it with the induced path metric in which edges  have length one. The group $\Aut(X)$ of graphical automorphisms of $X$ is thus a subgroup of  $\Isom(X)$, and is equipped with the subspace topology. Note that $\Aut(X)$ is a totally disconnected group, and is locally compact if $X$ is locally finite.

 The following fundamental theorem  is due to  van Dantzig:
\begin{thm}[\cite{vandantzig36}]\label{thm:vd}
	If $G$ is locally compact and totally disconnected, then every compact identity  neighbourhood  contains a compact open subgroup.
\end{thm}
	As noted in the introduction, \textbf{locally finite graphs are assumed to be both connected and singular}.  If a  graph $X$ is singular, then $\Aut(X)=\Isom(X)$. 
	Proposition \ref{prop:td_graph} and hence Theorem \ref{thm:mainintro} is false without this singularity hypothesis, since $\bbR$ is dominated by a non-singular graph, namely the Cayley graph of $\bbZ$ with generator $1$. \footnote{This is essentially the only counterexample, since every connected non-singular connected graph is homeomorphic to either $\bbR$ or $S^1$.}
\begin{prop}\label{prop:td_graph}
	If $G$ is locally compact and compactly generated, the following are equivalent:
	\begin{enumerate}
		\item $G$ is compact-by-(totally disconnected); \label{item:c-by-td}
		\item $G$ contains a compact open subgroup;\label{item:cpctopen}
		\item $G$ acts geometrically on a  locally finite vertex-transitive graph;\label{item:vtrangraph}
		\item $G$ acts geometrically on a locally finite  graph.\label{item:graph}
	\end{enumerate}
\end{prop}
\begin{proof}
	(\ref{item:c-by-td})$\implies$(\ref{item:cpctopen}):
	Suppose $G$  has a compact normal subgroup $K\vartriangleleft G$ such that $G/K$ is totally disconnected. Then Theorem \ref{thm:vd} ensures there exists a compact open subgroup $U/K\leq G/K$, whence $U$ is a compact open subgroup of $G$. 
	
	(\ref{item:cpctopen})$\implies$(\ref{item:vtrangraph}): This follows from the Cayley--Abels graph construction; see for instance \cite{kronmoller08roughcayley}.
	
	(\ref{item:vtrangraph})$\implies$(\ref{item:graph}) is obvious.
	
	(\ref{item:graph})$\implies$(\ref{item:c-by-td}): As $G$ acts geometrically on a singular locally finite  graph $X$, there is a copci map $\rho:G\rightarrow \Isom(X)$. As $X$ is singular, $\Isom(X)=\Aut(X)$ is totally disconnected and so $G$ is  compact-by-(totally disconnected).
\end{proof}
Combining Lemma \ref{lem:copci_action} and Proposition \ref{prop:td_graph}, we deduce:
\begin{cor}\label{cor:domin_by_loc_fin}
	If $X$ is a cocompact proper quasi-geodesic metric space, then $X$ is dominated by a locally finite  graph if and only if $\Isom(X)$ is compact-by-(totally disconnected).
\end{cor}

The following proposition  translates the problem of finding a model geometry of $\Gamma$ not dominated by a locally finite graph, to that of embedding $\Gamma$ as a uniform lattice in a suitable locally compact group:
\begin{prop}\label{prop:modgeom_vs_lattice}
	Let $\Gamma$ be a finitely generated group. The following are equivalent:
	\begin{enumerate}
		\item $\Gamma$ has a model geometry  not dominated by a locally finite  graph;\label{item:modgeom_vs_lattice_graph}
		\item There is a locally compact group $G$ and virtual uniform lattice embedding $\Gamma\to  G$ such that  $G^\circ$ is non-compact.\label{item:modgeom_vs_lattice_lattice}
	\end{enumerate}
\end{prop}
To prove this, we use the following lemma.
\begin{lem}[{\cite[Lemma 6, Corollary 7]{moshersageevwhyte2002maximally}}]\label{lem:quasi_geod}
	If $G$ is a locally compact compactly generated group, it has can be endowed with a left-invariant metric $(G,d)$ that is proper and quasi-geodesic. In particular, there exists a proper quasi-geodesic metric space $X$ such that $G$ acts  geometrically on $X$.
\end{lem}
\begin{proof}[Proof of Proposition \ref{prop:modgeom_vs_lattice}]
	(\ref{item:modgeom_vs_lattice_graph})$\implies$(\ref{item:modgeom_vs_lattice_lattice}): Suppose $X$ is a model geometry of $\Gamma$ not dominated by a locally finite graph. By Corollary \ref{cor:domin_by_loc_fin}, $\Isom(X)$ is not compact-by-(totally disconnected) or equivalently, $\Isom(X)^\circ$ is non-compact. Since $X$ is a model geometry, there is a geometric action  $\rho:\Gamma\to \Isom(X)$. Lemma \ref{lem:copci_action} and Proposition \ref{prop:discrete_copci} ensure $\rho$ is a virtual uniform lattice embedding.
	
	(\ref{item:modgeom_vs_lattice_lattice})$\implies$(\ref{item:modgeom_vs_lattice_graph}): Suppose $\rho:\Gamma\to G$ is a virtual uniform lattice embedding in a locally compact group $G$ with $G^\circ$ non-compact. By Proposition \ref{prop:discrete_copci}, $\rho$ is copci. Lemmas \ref{lem:copci_action} and \ref{lem:quasi_geod} ensure there is a proper quasi-geodesic metric space $X$ and a copci map $\phi:G\to \Isom(X)$. In particular, Lemma \ref{lem:dom} ensures $X$ is a model geometry of $\Gamma$. Since $G^\circ$ is non-compact, neither is the identity component of $\Isom(X)$. Therefore, Corollary \ref{cor:domin_by_loc_fin} implies $X$ is not dominated by a locally finite graph.
\end{proof}

We briefly mention a useful consequence of van Dantzig's theorem.
\begin{prop}\label{prop:vand+nss}
	Let $G$ be a Lie group. Any totally disconnected  subgroup $H\leq G$ is discrete. In particular, if $H$ is both totally disconnected and compact, then it is finite.
\end{prop}
\begin{proof}
Lie groups satisfy the \emph{no small subgroups} property: there is a sufficiently small identity neighbourhood $U\subseteq G$ containing no non-trivial subgroups. In particular, the only subgroup contained in $H\cap U$ is trivial.  By Theorem \ref{thm:vd}, $H\cap U$ contains a subgroup that is  open in $H$, hence $\{1\}$ is open in $H$ and so $H$ is discrete.
\end{proof}

\section{Engulfing groups  and uniformly commensurated subgroups}\label{sec:engulf}

Let $\Lambda$ be a group. A \emph{commensuration} of $\Lambda$ is an isomorphism $\Lambda_1\to\Lambda_2$ between finite index subgroups of $\Lambda$. The \emph{abstract commensurator} $\Comm(\Lambda)$ of  $\Lambda$ is the group of all equivalence classes of commensurations of $\Lambda$, where two commensurations   are equivalent if they agree on a common finite index subgroup of $\Lambda$. The group operation on $\Comm(\Lambda)$ is composition after restricting the domain to a finite  index subgroup so that composition is defined.

\begin{exmp}\label{exmp:abs_comm}
	The abstract commensurator of $\bbZ^n$ is $\GL_n(\bbQ)$. More specifically, identifying  $\bbZ^n$ with the integer lattice in $\bbR^n$, the natural action of  $\GL_n(\bbQ)$ on $\bbR^n$ acts as commensurations of $\bbZ^n$, and every commensuration of $\bbZ^n$ arises in this way.
\end{exmp}

If $\Gamma$ is a group and $\Lambda\alnorm \Gamma$ is a commensurated subgroup, then there is a  homomorphism $\Phi_{\Gamma,\Lambda}:\Gamma\to \Comm(\Lambda)$ induced by conjugation. Indeed, for every $g\in G$ the  isomorphism $\phi_g:\Lambda\cap g^{-1}\Lambda g\to g\Lambda g^{-1}\cap \Lambda$  given by $h\mapsto ghg^{-1}$ is  a commensuration of $\Lambda$. The map $\Phi_{\Gamma,\Lambda}:\Gamma\to \Comm(\Lambda)$  given by $g\mapsto [\phi_g]$ is called the \emph{modular homomorphism}.

If $X$ is a metric space, the \emph{quasi-isometry group} $\QI(X)$  is the set of all equivalence classes of self quasi-isometries of $X$, where two quasi-isometries  $f,g:X\to X$ are equivalent if $\sup_{x\in X}d(f(x),g(x))<\infty$. The group operation on $\QI(X)$ is induced by composition. If $X$ and $Y$ are quasi-isometric, then a quasi-isometry $X\to Y$ naturally induces an isomorphism  $\QI(X)\to\QI(Y)$. A subset $\Omega\subseteq \QI(X)$ is \emph{uniform} if there exist constants $K$ and $A$ such that each equivalence class in $\Omega$ contains a $(K,A)$-quasi-isometry.

Let $\Lambda$ be a finitely generated group equipped with the word metric. For every commensuration  $\phi:\Lambda_1\to\Lambda_2$  of $\Lambda$,  the map $\phi^*=i_{\Lambda_2}\circ\phi\circ q_{\Lambda_1}:\Lambda\to \Lambda$ is a quasi-isometry, where $i_{\Lambda_2}:\Lambda_2\to \Lambda$ and  $q_{\Lambda_1}:\Lambda\to \Lambda_1$  are the inclusion map  and a closest point projection respectively.  This induces  a homomorphism $\Psi_\Lambda:\Comm(\Lambda)\to\QI(\Lambda)$ given by $[\phi]\to[\phi^*]$, which is independent of the choice of commensuration and closest point projection.  Recall from the introduction that a finitely generated commensurated subgroup $\Lambda\alnorm \Gamma$ is \emph{uniformly commensurated} if the image of the map \[\Gamma\xrightarrow{\Phi_{\Gamma,\Lambda}} \Comm(\Lambda)\xrightarrow{\Psi_\Lambda} \QI(\Lambda)\] is uniform.

\begin{rem}\label{rem:induced_qi}
	If $\Lambda$ is a finitely generated group, the natural map  $\Lambda\to \QI(\Lambda)$ induced by the action of $\Lambda$ on itself by left-multiplication, coincides with the composition $\Lambda\to\Comm(\Lambda)\to\QI(\Lambda)$.
\end{rem}

A \emph{quasi-action} of a group $G$ on a space $X$ is a map $\phi:G\to X^X$ such that there exists constants $K\geq 1$ and $A\geq 0$ such that the following hold:
\begin{itemize}
	\item for all $g\in G$, $\phi(g)$ is  a $(K,A)$-quasi-isometry;
	\item for all $g,h\in G$ and $x\in X$, $d(\phi(gh)(x),\phi(g)(\phi(h)(x)))\leq A$;
	\item for all $x\in X$, $d(x,\phi(1)(x))\leq A$.
\end{itemize}We typically suppress the map $\phi$ and write a quasi-action $\phi:G\to X^X$ as $G\qa X$. Two quasi-actions $G\qa X$ and $G\qa Y$ are \emph{quasi-conjugate} if there is a quasi-isometry $f:X\to Y$ that is coarsely uniformly $G$-equivariant.

A metric space $X$ is \emph{tame} if for every $K$ and $A$, there exists a constant $B$ such if two $(K,A)$-quasi-isometries from $X$ to $X$ are close, they are $B$-close. We will use the following theorem of Kapovich--Kleiner--Leeb, which is a combination of  Propositions 4.8 and 5.9 in \cite{kapovich1998derham}:
\begin{prop}[{\cite[Propositions 4.8 and 5.9]{kapovich1998derham}}]
	Symmetric spaces of non-compact type are tame.
\end{prop}

The following lemma  will be used as a source of quasi-actions:
\begin{lem}\label{lem:induced_qa}
	If $X$ is  tame and $\rho:G\to\QI(X)$ is a homomorphism with uniform image, then there is  a quasi-action $\phi:G\to X^X$ such that \[[\phi(g)]=\rho(g).\]
\end{lem}  
\begin{proof}
	Since $\rho$ has a uniform image, there exist $K$ and $A$ such that for each $g\in G$, there is some $(K,A)$-quasi-isometry $\phi(g):X\to X$ such that $[\phi(g)]=\rho(g)$. As $\rho$ is a homomorphism, for all $g,h\in G$ we have $[\phi(g)\phi(h)]=[\phi(gh)]$. Since $X$ is tame, there is a constant $B$, depending only on $(K,A)$ such that $\phi(g)\phi(h)$ and $\phi(gh)$ are $B$-close for all $g,h\in A$, and $\phi(1)$ is $B$-close to the identity. This shows $\phi$ is a quasi-action.
\end{proof}

A symmetric space $X$ with de Rham decomposition $X_1\times \dots \times X_n$ is said to be \emph{normalised} if whenever $X_i$ and $X_j$ are homothetic, they are isometric.  We recall the following theorem of Kleiner--Leeb:

\begin{thm}[\cite{kleinerleeb09quasiactions}]\label{thm:quasi-action}
	Let $X$ be a normalised symmetric space of non-compact type  and $G\qa X$ be a quasi-action. Then $G$ is quasi-conjugate to an isometric action  on $X$.
\end{thm}

Given a group $\Gamma$ containing  an infinite cyclic commensurated subgroup $\Lambda= \langle a \rangle\cong \bbZ$,  $\Gamma$ does not  act on $\Lambda$ by conjugation unless $\Lambda$ is normal. To remedy this, we embed $\Lambda$ as a lattice in $\bbR$ and define an action of $\Gamma$ on $\bbR$ as follows. If $g\in \Gamma$ and $ga^mg^{-1}=a^n$ for $n,m\neq 0$, then $g$ acts on $\bbR$ via multiplication by  $\frac{n}{m}$. Note that as $\Lambda$ is commensurated, such $m$ and $n$ exist. This gives a well-defined action $\Delta:\Gamma\rightarrow \Aut(\bbR)$ such that $\Delta(g)$ agrees with  conjugation by $g$   on a finite index subgroup of $\Lambda$. This phenomenon motivates the following definition:
\begin{defn}\label{defn:engulf}
	A \emph{Hecke pair} $(\Gamma,\Lambda)$ consists of a group $\Gamma$ and a commensurated subgroup $\Lambda\alnorm \Gamma$. We say that a locally compact group $E$ is an \emph{engulfing group} of the Hecke pair $(\Gamma,\Lambda)$ if  there is a virtual uniform lattice embedding $\rho:\Lambda\rightarrow E$ and a  homomorphism $\Delta:\Gamma\to \Aut(E)$ such that  \begin{enumerate}
		\item $\Delta$ extends the map $\Lambda\xrightarrow{\rho} E\to  \Aut(E)$, where  $E\to \Aut(E)$ is the action of $E$ on itself by  left conjugation.
		\item for all $g\in \Gamma$, there is a finite index subgroup $\Lambda_g\leq \Lambda$ such that  \[\Delta(g)(\rho(h))=\rho(ghg^{-1})\] for all $h\in \Lambda_g$.
	\end{enumerate} More generally, we say $(E,\rho,\Delta)$ as above is an \emph{engulfing triple} of the Hecke pair $(\Gamma,\Lambda)$.
\end{defn}
Engulfing groups play a central role in Theorem \ref{thm:build_lattice_embedding}, which is used to prove one direction of Theorem \ref{thm:mainintro}. The following proposition generalises the discussion preceding Definition \ref{defn:engulf}.
\begin{prop}\label{prop:engulf_abelian}
	Let $(\Gamma,\Lambda)$ be a Hecke pair with $\Lambda$ isomorphic to $\bbZ^n$. Then $(\Gamma,\Lambda)$ has engulfing group $\bbR^n$.
\end{prop}
\begin{proof}
	We fix an embedding $\rho:\Lambda\to \bbR^n$  with image  the integer lattice and identify $\Lambda$ with its image. As noted in Example \ref{exmp:abs_comm}, all commensurations of $\Lambda$ arise from the standard action of an element $\GL_n(\bbQ)$  restricted to a finite index subgroup of $\Lambda$.  The modular homomorphism $\Gamma\to \Comm(\Lambda)$ thus induces a homomorphism $\Delta:\Gamma\to \GL_n(\bbQ)\leq \GL_n(\bbR)=\Aut(\bbR^n)$.  By construction,  for each $g\in G$ there is a finite index subgroup $\Lambda_g\leq \Lambda$ such that  $\Delta(g)(\rho(h))=\rho(ghg^{-1})$ for all $h\in \Lambda_g$. Thus $\bbR^n$ is an engulfing group of $(\Gamma,\Lambda)$.
\end{proof}
\begin{rem}
	Using a result of Malcev \cite{malcev1951class}, an analogue of Proposition \ref{prop:engulf_abelian} holds when $\Lambda$ is a finitely generated torsion-free nilpotent group.
\end{rem}

We have the following sufficient criterion for  $E$ to be an engulfing group.
\begin{lem}\label{lem:engulfing}
		Let $(\Gamma,\Lambda)$ be a Hecke pair with  $\Lambda$  finitely generated. Suppose there is a locally compact group $E$ and a homomorphism $\phi:\Gamma\to E$ such that $\phi|_\Lambda$ is a virtual uniform lattice embedding. Then $E$ is an engulfing group of $(\Gamma,\Lambda)$ and $\Lambda$ is uniformly commensurated.
\end{lem}
\begin{proof}
We define  $\Delta:\Gamma\to\Aut(E)$ by $\Delta(g)(l)=\phi(g)l\phi(g)^{-1}$ and set $\rho\coloneqq \phi|_\Lambda$. Thus $\Delta(g)\phi(h)=\phi(ghg^{-1})$ for all $g,h\in \Gamma$ and so $(E,\rho,\Delta)$ is an engulfing triple of $(\Gamma,\Lambda)$.
We equip $\Lambda$ and $E$ with the word metrics $d_\Lambda$ and $d_E$ with respect to finite and  compact generating sets respectively. By Lemma  \ref{lem:quasi-isom_copci} and Proposition \ref{prop:discrete_copci},  $\rho:\Lambda\to E$ is a quasi-isometry. Let  $f:E\to \Lambda$ be a coarse inverse to $\rho$. For $e\in E$, let $L_e:E\to E$ be left multiplication by $e$, which is an isometry of $E$.  Then there exist constants $K$ and $A$ such for every $g\in \Gamma$, the map $f_{g}\coloneqq f\circ L_{\phi(g)}\circ \rho:\Lambda\to \Lambda$ is a $(K,A)$-quasi-isometry. We may also assume, by increasing $K$ and $A$ if needed, that $f$ and $\rho$ are $(K,A)$-quasi-isometries and $ \rho\circ f$ is $A$-close to the identity.
Then for all $g\in G$ and  $h\in g^{-1}\Lambda g\cap \Lambda$, we have
\begin{align*}
 d_\Lambda(ghg^{-1},f_g(h))	&\leq  Kd_E(\rho(ghg^{-1}),\rho(f(\phi(g)\rho(h))))+KA\\
 &\leq  Kd_E(\phi(ghg^{-1}),\phi(gh))+2KA\\
 &=  Kd_E(\phi(g^{-1}),1)+2KA,
\end{align*}
noting that this bound depends only on $g$ and is independent of $h$. This ensures the  map $\Gamma\to\Comm(\Lambda)\to\QI(\Lambda)$ is given by $g\mapsto [f_g]$. Since every $f_g$ is a $(K,A)$-quasi-isometry,  $\Lambda$ is uniformly commensurated.
\end{proof}

The following proposition gives a situation in which Lemma \ref{lem:engulfing} can be applied.
\begin{prop}\label{prop:engulf_commen_lattice}
	Let $\Gamma$ be a  group and let $\Lambda\alnorm \Gamma$ be a uniformly commensurated subgroup that is a virtual uniform lattice in a connected centre-free semisimple Lie group. 		Then there is a centre-free semisimple Lie group $S$ with  no compact factors and finitely many components, and a   homomorphism $\phi:\Gamma\to S$  such that $\phi|_{\Lambda}$ is a virtual uniform lattice embedding. In particular,  $S$ is an engulfing  of the pair $(\Gamma,\Lambda)$.
\end{prop}

\begin{proof}
First note that  $\Lambda$ acts  geometrically on a symmetric space $X$ of non-compact type; see Proposition \ref{prop:manifold_gp_action}. Without loss of generality, we may assume $X$ is normalised. Since $X$ is tame and   $\Lambda$ is uniformly commensurated,  Lemma \ref{lem:induced_qa} implies the  map $\Gamma\to \Comm(\Lambda)\to \QI(\Lambda)\cong \QI(X)$ induces a quasi-action $\Gamma\qa X$. By Theorem \ref{thm:quasi-action}, this quasi-action is quasi-conjugate to an isometric action $\phi:\Gamma\to \Isom(X)$. By Remark \ref{rem:induced_qi}, this restricts to an isometric action $\phi|_\Lambda$ that is quasi-conjugate to the natural action of $\Lambda$ on itself, hence is geometric. By Lemma \ref{lem:copci_action} and Proposition \ref{prop:discrete_copci}, $\phi|_\Lambda$ is a  virtual uniform lattice embedding.   Lemma \ref{lem:engulfing} now ensures $\Isom(X)$ is an engulfing  of the pair $(\Gamma,\Lambda)$. Moreover, Proposition \ref{prop:manifold_gp_action} ensures $\Isom(X)$ is a semisimple Lie group of the required type.
\end{proof}
We can now use Proposition \ref{prop:engulf_commen_lattice} to describe the structure of uniformly commensurated normal subgroups that are lattices in semisimple Lie groups. 
\begin{cor}\label{cor:unif_comm}
	Let $\Lambda\vartriangleleft \Gamma$ be a uniformly commensurated normal subgroup that is isomorphic to a uniform lattice in a connected  centre-free semisimple Lie group with no compact factors and finitely many components. Then there is a normal subgroup $\Delta\vartriangleleft \Gamma$ such that some  finite index subgroup of $\Gamma$  splits as a direct product $\Lambda\times \Delta$. 
\end{cor}
\begin{proof}
	The hypothesis on $\Lambda$ ensures it contains no finite normal subgroups. By Proposition \ref{prop:engulf_commen_lattice}, there is a centre-free semisimple Lie group $S$ with no compact factors and finitely many components, and a   map $\phi:\Gamma\to S$ such that $\phi|_\Lambda$ is a virtual uniform lattice embedding. Since $\Lambda$ has no finite normal subgroups, $\phi|_\Lambda$ is injective. 
	
	 As $\phi|_\Lambda$ is injective, $\Lambda\cap\ker(\phi)=\{1\}$. Let  $g\in \ker(\phi)$ and $h\in \Lambda$. Since $\phi(g)=1$, $\phi([g,h])=1$. As $\Lambda$ is normal, $[g,h]\in \Lambda$ and so injectivity of $\phi|_\Lambda$  ensures $[g,h]=1$.  Since $\ker(\phi)$ and $\Lambda$ commute and have trivial intersection,  $\Gamma_1\coloneqq \phi^{-1}(\phi(\Lambda))=\ker(\phi)\Lambda$ is isomorphic to the direct product $\ker(\phi)\times \Lambda$. As $\Lambda$ is normal in $\Gamma$, $\phi(\Gamma)\subseteq N_{S}(\phi(\Lambda))$. However, since $\phi(\Lambda)$ is a uniform lattice in $S$, a consequence of Borel density ensures $\phi(\Lambda)$ is of finite index in its normaliser $N_{S}(\phi(\Lambda))$ \cite[Theorem 5.26]{raghunathan1972discrete}, implying $\Gamma_1$ is a finite index subgroup of $\Gamma$.
\end{proof}

We digress briefly to discuss Mostow rigidity, giving examples of commensurated subgroups that are automatically uniformly commensurated. 
Let $G$ be a centre-free semisimple Lie group with no compact factors and finitely many components. Then $G$ splits as a direct product $G_1\times \dots \times G_n$, where each $G_i$ is simple and non-compact.  A uniform lattice $\Gamma\leq G$ is said to be \emph{Mostow rigid} if  there is no  factor $G_i$ isomorphic to $\PSL(2,\bbR)$ such that $\Gamma G_i\leq G$ is closed. Equivalently, $\Gamma$ does not virtually split as a direct product with some direct factor a uniform lattice in  $\PSL(2,\bbR)$.
\begin{thm}[{\cite[Theorem A']{mostow1973strong}}]\label{thm:mostowrigid}
	If $G$ is a connected centre-free semisimple Lie group with no compact factors and $\Gamma,\Gamma'\leq G$ are two Mostow rigid lattices, then any isomorphism $\Gamma\to \Gamma'$ extends uniquely to a smooth automorphism $G\to G$.
\end{thm}
We can use this rigidity to show the following.
\begin{prop}\label{prop:unifcomm_mostow}
	Let $(\Gamma,\Lambda)$ be a Hecke pair and let $S$ be a connected centre-free semisimple Lie group $S$ with  no compact factors. Suppose $\rho:\Lambda\to S$ is a virtual uniform lattice embedding with  Mostow rigid image.  Then $\Lambda$ is a uniformly commensurated subgroup of $\Gamma$.
\end{prop}
 We first recall the structure of   automorphism groups of semisimple Lie groups. Let $G$ be a connected centre-free semisimple Lie group with Lie algebra $\mathfrak{g}$. Since $G$ is a connected, there is an injective homomorphism $d:\Aut(G)\to \Aut(\mathfrak{g})$. Because $G=\widetilde{G}/Z(\widetilde{G})$ and $Z(\widetilde{G})$ is characteristic and discrete, the  map $d:\Aut(G)\to \Aut(\mathfrak{g})$ is  surjective hence an isomorphism.   Thus $\Aut(G)$ can be given the structure of a Lie group; for details see \cite[\S1.2.10]{onishchikvinberg1990lie}. 
 
 As $G$ is centre-free, the adjoint map $\Ad:G\to \Aut(G)$ is injective.  Moreover, $\Ad(G)=\Inn(G)$ maps isomorphically to $\Inn(\mathfrak{g})$. Since $\mathfrak{g}$ is semisimple, $\Aut(\mathfrak{g})$ has finitely many components and  $\Inn(\mathfrak{g})$ is the identity component of $\Aut(\mathfrak{g})$; see for instance \cite{murakami1952automorphisms}. Thus $G\cong\Aut(G)^\circ$. Note  $Z(\Aut(G))=1$ since every element in $Z(\Aut(G))$ acts trivially on $G\cong\Aut(G)^\circ$ by conjugation, hence is the identity. To summarise:
\begin{lem}\label{lem:fin_cen}
	Let $G$ be a connected centre-free semisimple Lie group. Then $\Aut(G)$ is also a centre-free semisimple Lie group with  finitely many components. Moreover,  $\Ad:G\to \Aut(G)$ is a Lie group isomorphism onto $\Inn(G)=\Aut(G)^\circ$.
\end{lem}

\begin{proof}[Proof of Proposition \ref{prop:unifcomm_mostow}]
	Let $\rho:\Lambda\to S$ be a virtual uniform lattice embedding with Mostow rigid image. For every finite index $\Lambda'\leq \Lambda$, $\rho(\Lambda')$ is also Mostow rigid.  By Theorem \ref{thm:mostowrigid}, for each $g\in \Gamma$, there is an automorphism $\Delta(g):S\to S$ such that $\Delta(g)(\rho(h))=\rho(ghg^{-1})$ for all $h\in \Lambda\cap g^{-1}\Lambda g$. For all $g,k\in \Gamma$, $\Delta(gk)$ and $\Delta(g)\Delta(k)$ agree on a finite index subgroup of $\Lambda$, hence by    Theorem \ref{thm:mostowrigid} we have $\Delta(gh)=\Delta(g)\Delta(k)$. Thus $\Delta:\Gamma\to \Aut(S)$ is a homomorphism.  Lemma \ref{lem:fin_cen} ensures  the image of the map \[\Delta|_\Lambda:\Lambda\xrightarrow{\rho} S\xhookrightarrow{\Ad} \Aut(S)\] is a uniform lattice. Applying Proposition \ref{prop:engulf_commen_lattice} to $\Delta$, we deduce $\Lambda$ is uniformly commensurated in $\Gamma$.
\end{proof}

\section{The structure of lattices in locally compact groups}\label{sec:structure}
In this section we prove Theorems \ref{thm:modelgeom}, \ref{thm:lattice_prod} and one direction of Theorem \ref{thm:mainintro}. These will  follow from Proposition \ref{prop:lattice_structure}. The following major result  forms part of the solution to Hilbert's fifth problem:

\begin{thm}[\cite{gleason1951structure,yamabe1953generalization}]\label{thm:gleason-yamabe}
	Let $G$ be a connected-by-compact locally compact group. Then there is a compact subgroup $K\vartriangleleft G$ such that $G/K$ is a Lie group with finitely many components.
\end{thm}
A useful consequence of Theorem \ref{thm:gleason-yamabe} is the following, which allows us to assume, after quotienting out by a compact normal subgroup of $G$, that $G^\circ$ has no non-trivial compact normal subgroup.
\begin{lem}\label{lem:conn-comp-lie}
	Let $G$ be a locally compact group. There is a compact normal subgroup $K\vartriangleleft G$ such that   $(G/K)^\circ$ is a connected Lie group  with  no non-trivial compact normal subgroups. Moreover, $G^\circ$ is compact if and only if $(G/K)^\circ$ is compact.
\end{lem}
\begin{proof}
	Since $G^\circ$ is connected, Theorem \ref{thm:gleason-yamabe} ensures it contains a compact subgroup $K_0\vartriangleleft G^\circ$ such that $G^\circ/K_0$ is a connected Lie group. 
	As  $G^\circ/K_0$ is a connected Lie group, Theorem \ref{thm:maxcompact} ensures it  has a maximal compact subgroup  $L/K_0$ such that  every compact subgroup of $G^\circ /K_0$ is contained in a conjugate of $L/K_0$. Let $K/K_0$ be the normal core of $L/K_0$. Then $K$ is a maximal  compact normal subgroup of $G^\circ$, i.e.\ it contains every other compact normal subgroup of $G^\circ$. In particular, $K$ is topologically characteristic in $G^\circ$ and therefore normal in $G$. 
	
	Let $\pi:G\rightarrow G/K$ be the quotient map and set $H\coloneqq \pi^{-1}((G/K)^\circ)$. As $\pi(G^\circ)$ is connected, $G^\circ\leq H$. The quotient $H/G^\circ$ is totally disconnected by Proposition \ref{prop:idcomp}, and is also the quotient of the connected group $H/K=(G/K)^\circ$, hence is connected. Thus $H/G^\circ$ is trivial and so $G^\circ=H$. Since $K$ is a maximal compact normal subgroup of $G^\circ$, $G^\circ/K=(G/K)^\circ$ contains no non-trivial compact normal subgroups, hence by Theorem \ref{thm:gleason-yamabe}, $(G/K)^\circ$ is a connected Lie group.
	As $K$ is compact and $G^\circ/K=(G/K)^\circ$, we see $G^\circ$ is compact if and only if  $(G/K)^\circ$ is compact.
\end{proof}

A connected Lie group $G$ has a \emph{Levi decomposition} $G=RS$, where:
	\begin{enumerate}
		\item $R$ is the unique maximal closed connected normal solvable subgroup of $G$ known as the \emph{radical} of $G$;
		\item $S$ is a closed connected semisimple subgroup of $G$, unique up to conjugation, known as the \emph{Levi subgroup} of $G$.
	\end{enumerate}
The \emph{nilradical} of a connected Lie group $G$  consists of the  unique maximal closed connected normal nilpotent subgroup of $G$. Every  Lie group  with non-trivial radical  has non-trivial nilradical. 
Engel's theorem implies a maximal compact subgroup $K$ of a connected nilpotent Lie group $N$ is necessary central in $N$ (\cite[Chapter 1, Theorem 1.6]{vinbergonischik}), hence $K$ is a compact characteristic subgroup of $N$. 
Thus if $G$ has no non-trivial compact normal subgroup, then a maximal torus of its nilradical $N$ is trivial, hence $N$ is simply connected.    We thus deduce:
\begin{prop}\label{prop:nilrad}
	Let $G$ be a connected Lie group with no non-trivial compact normal subgroup. Then exactly one of the following occurs:
	\begin{enumerate}
		\item the nilradical $N$ of $G$ is non-trivial and simply connected;
		\item $G$ is a centre-free semisimple Lie group.
	\end{enumerate}
\end{prop}

The following proposition partially mitigates the fact that $G^\circ $ will not typically be open.

\begin{prop}\label{prop:opensbgrp_prod}
	Let $G$ be a locally compact group such that $G^\circ$ contains no non-trivial compact normal subgroup. There is an  open subgroup $L\leq G$ and a compact totally disconnected subgroup $K\vartriangleleft L$ such that:
	\begin{enumerate}
		\item the map $G^\circ \times K\to L$  induced by the inclusions $G^\circ\to L$ and $K\to L$ is a topological isomorphism;
		\item for any topologically characteristic subgroup $M\leq G^\circ$, $M\times K\cong  MK\leq G$ is a commensurated subgroup of $G$.
	\end{enumerate}
\end{prop}
\begin{proof}
	Since $G^\circ$ has no non-trivial compact normal subgroup, it is necessarily a Lie group  by Theorem \ref{thm:gleason-yamabe}. Let $\pi:G\rightarrow G/G^\circ$ be the quotient map. By Proposition \ref{prop:idcomp} and Theorem \ref{thm:vd}, there is a compact open subgroup $U\leq G/G^\circ$. Let $L\coloneqq \pi^{-1}(U)$. Thus $L$ is open and  connected-by-compact, hence by Theorem \ref{thm:gleason-yamabe}, $L$ contains a compact normal subgroup $K$ such that $L/K$ is a virtually connected Lie group. Replacing $L$ with a finite index open subgroup if needed, we may assume $L/K$ is connected. We claim $L$ is topologically isomorphic to $G^\circ \times K$.
	
	 Since $K$ is compact and $G^\circ$ is closed, $G^\circ K$ is closed. Since $L/G^\circ K$ is the quotient of the connected group $L/K$ and the totally disconnected group $L/G^\circ$, we deduce $L=G^\circ K$. Since $G^\circ \cap K$ is a compact normal subgroup of $G^\circ$, $G^\circ \cap K=\{1\}$. Since $L=G^\circ K$, $G^\circ \cap K=\{1\}$ and both $G^\circ$ and  $K$ are closed and  normal in $L$, they commute and so there is a continuous abstract isomorphism  $\phi:G^\circ \times K\rightarrow L$ induced by the inclusion in each factor.	 
	 Let $\pi:L\rightarrow L/G^\circ$ be the quotient map. The restriction $ \pi|_K=\psi:K\rightarrow L/G^\circ$ is  a continuous abstract isomorphism from a compact space to a Hausdorff space, so is a topological isomorphism.  This ensures $\phi$ has a topological inverse, hence $L$ is topologically isomorphic to $G^\circ \times K$ (see \cite[\S III.2.10]{bourbaki98gentop1-4}).
	 
	 We claim $K$ is a commensurated subgroup of $G$. Fix $g\in G$.  Let $\psi:L\to L/K\cong G^\circ$ be the quotient map. As $L$ is open and  $g^{-1}Kg$ is compact, $g^{-1}Kg\cap L$ is a finite index subgroup of $g^{-1}Kg$. Moreover,  $\psi(g^{-1}Kg\cap L)$ is a compact totally disconnected subgroup of the Lie group  $G^\circ$, hence is finite by Proposition \ref{prop:vand+nss}. Therefore $g^{-1}Kg\cap K$ is a finite index subgroup of $g^{-1}Kg$. Since this is true for all $g\in G$, it follows $K$ is commensurated. 	 
	 Finally, let $M$ be a topologically characteristic subgroup of $G^\circ$ and let $g\in G$. As $K$ is commensurated and  $gMg^{-1}=M$, we see $g(MK)g^{-1}=MgKg^{-1}$ is commensurable to $MK$, hence $M K\cong M\times K$ is a commensurated subgroup of $G$ as required.
\end{proof}

We use Lemma \ref{lem:fin_cen} and Proposition \ref{prop:opensbgrp_prod} to deduce the following useful proposition. 
\begin{prop}\label{prop:map_to_product}
	Let $G$ be a locally compact group such that $G^\circ$ is a centre-free semisimple Lie group with no compact factors. Then there is a continuous open monomorphism $\phi:G\to \Aut(G^\circ) \times G/G^\circ$ with finite index image. In particular, $\phi$ is copci.
\end{prop}

\begin{proof}
	There is a continuous homomorphism  $\phi_1:G\to \Aut(G^\circ)$  induced by conjugation.  Lemma \ref{lem:fin_cen} implies that $\phi_1|_{G^\circ}=\Ad$ is a Lie group isomorphism  $G^\circ\to\Inn(G^\circ)=\Aut(G^\circ)^\circ$. Therefore, if $U\subseteq G$ is open, then so is \[\phi_1(U)=\bigcup_{g\in G}\phi_1(U\cap g G^\circ)=\bigcup_{g\in G}\phi_1(g)\phi_1(g^{-1}U\cap G^\circ).\] Thus $\phi_1$ is an open map.
	Let $\phi_2:G\to G/G^\circ$ be the quotient map, which is continuous and open, and let $\phi\coloneqq (\phi_1,\phi_2):G\to \Aut(G^\circ) \times G/G^\circ$.	
	Since $\phi_1$ and $\phi_2$ are continuous, so is $\phi$. Lemma \ref{lem:fin_cen} implies  $\phi_1(G^\circ)$  is a finite index  subgroup of $\Aut(G^\circ)$, hence so is $\phi_1(G)$. Since $\phi_2$ is surjective,   $\phi(G)$ is a finite index subgroup of $\Aut(G^\circ) \times G/G^\circ$. Since $\phi_1|_{G^\circ}$ is injective and $\ker(\phi_2)=G^\circ$,  $\phi$ is injective.
	
	All that remains is to show $\phi$ is an  open map. Let $U\subseteq G$ be open and let $g\in U$. Let $K$ and  $L\cong G^\circ\times  K$ be as in Proposition \ref{prop:opensbgrp_prod}. Since $g^{-1}U\cap L$ is open and contains the identity, there exists an identity neighbourhood contained in $g^{-1}U\cap L$ of the form $U_1U_2\cong U_1\times U_2$, where $U_1\subseteq G^\circ $ and $U_2\subseteq K$. Since $U_2$ commutes with $G^\circ$, we deduce that  $\phi_1(gU_1U_2)=\phi_1(gU_1)$ and $\phi_2(gU_1U_2)=\phi_2(gU_2)$. We claim $\phi_1(gU_1)\times \phi_1(gU_2)\subseteq \phi(U)$. To see this, let $(x,y)\in \phi_1(gU_1)\times \phi_1(gU_2)$. Thus $x=\phi_1(gu_1)$ and $y=\phi_2(gu_2)$ for some $u_1\in U_1$ and $u_2\in U_2$, and so $\phi(gu_1u_2)=(x,y)\in \phi(U)$. Therefore $\phi_1(gU_1)\times \phi_1(gU_2)\subseteq \Aut(G^\circ) \times G/G^\circ$ is open neighbourhood of $\phi(g)$ contained in $\phi(U)$. Thus $\phi(U)$ is open, hence $\phi$ is an open map.	
\end{proof}

The following lemma ensures that quotienting out by a compact normal subgroup preserves lattices.
\begin{lem}[{\cite[Theorem 4.7, Corollary 4.10]{vinberggorbatsevichshvartsman2000discrete}}]\label{lem:lattice_lemma_quotient}
	Let $G$ be a locally compact group and $\Gamma\leq G$ be a (uniform) lattice. If $K\vartriangleleft G$ is a compact normal subgroup of $G$ and $\pi:G\rightarrow G/K$ is the quotient map, then $\Gamma/(\Gamma\cap K)\cong \pi(\Gamma)$ is a (uniform) lattice in $G/K$.
\end{lem}

The following two lemmas demonstrate that the  intersection of  a lattice in $G$ with certain   subgroups  $H\leq G$ are  lattices in $H$. Such subgroups are called \emph{lattice-hereditary}.
 
\begin{lem}[{see \cite[\S 2.C]{capracemonod2012lattice}}]\label{lem:lattice_open}
	Let $G$ be a locally compact group and  $\Gamma\leq G$ be a (uniform) lattice. If $H\leq G$ is open, then $H\cap \Gamma$ is a (uniform) lattice in $H$.
\end{lem}

\begin{lem}\label{lem:mostow_lattice_series}
	Let $G$ be a connected Lie group with no non-trivial compact normal subgroup  and let $\Gamma\leq G$ be a lattice. Let $N$ be the nilradical of $G$ and let $C_k(N)$ be the $k$th term in the lower central series of $N$. Then $C_k(N)\cap \Gamma$ is a uniform lattice in $C_k(N)$.
\end{lem}
\begin{proof}
	 A lemma of Mostow states that if $N\leq G$ is the nilradical of $G$, then  the intersection $N\cap \Gamma$ is a uniform lattice in $N$ \cite[Lemma 3.9]{mostow1971arithmetic}. If $N=\{1\}$ there is nothing to prove. If not, then we are in Case (2) of Proposition \ref{prop:nilrad}, and so $N$ is simply connected. Since $\Gamma\cap N$ is a lattice in $N$,  $C_k(N)\cap \Gamma$ is a uniform lattice in $C_k(N)$ by  a result of Malcev \cite{malcev1951class}; see  \cite[Chapter 2, Corollary 2.5]{vinbergonischik} for details.
\end{proof}
\begin{rem}
	The problem of determining  lattice-hereditary subgroups of Lie groups has ``a long and rather dramatic history''  according to Gorbatsevich's  MathSciNet review of \cite{geng2015When}.   We refer the reader to this article of Geng for a discussion of the topic and a summary of what is known, noting that several disputed claims appear in the literature  \cite{geng2015When}.
\end{rem}

We will also make use of the following lemma to show commensurated subgroups are preserved under images and preimages.
\begin{lem}[{\cite[Lemmas 3.7 and 3.8]{connermihalik2013}}]\label{lem:preimage_comm}
	Suppose $G,H$ are groups and  $\phi:G\rightarrow H$ is a homomorphism.
	\begin{enumerate}
		\item If $K\alnorm H$ is commensurated, then $\phi^{-1}(K)\alnorm G$ is commensurated.
		\item If $L\alnorm G$ is commensurated and $\phi$ is surjective, then $\phi(L)$ is commensurated.
	\end{enumerate} 
\end{lem}

The following lemma follows easily  from the fact that the centre of a finitely generated group $\Gamma$ has finite index when the  commutator subgroup $\Gamma'$ is finite.
\begin{lem}\label{lem:finite_index_abelian}
	Let $\Gamma$ be a finitely generated group containing a finite  normal subgroup $F\vartriangleleft \Gamma$ such that $\Gamma/F\cong \bbZ^n$. Then $\Gamma$ contains  a finite index subgroup isomorphic to $\bbZ^n$.
\end{lem}

We now come to the main technical result of this section.

\begin{prop}\label{prop:lattice_structure}
	Let $\Gamma$ be a finitely generated group, $G$ be a locally compact group,  and suppose $\rho:\Gamma\to G$ is a virtual lattice embedding. If $G^\circ$ is non-compact, then $\Gamma$ contains an infinite commensurated subgroup $\Lambda\leq \Gamma$ such that either:
	\begin{enumerate}
		\item $\Lambda$ is a finite rank  free abelian group;
		\item  There is a copci homomorphism $\tau:G\to S\times D$ with  finite index  open image such that:
		\begin{enumerate}
			\item $S$ is a centre-free semisimple Lie group with   no compact factors and finitely many components
			\item $D$ is a compactly generated totally disconnected locally compact group
			\item The composition \[\Lambda\xrightarrow{\iota} \Gamma\xrightarrow{\rho}G\xrightarrow{\tau}S\times D\xrightarrow{q}S\] is a virtual lattice embedding, where $\iota$ is the inclusion and $q$ is the projection.
		\end{enumerate} 
	Moreover,  if $\rho$ is a uniform virtual lattice embedding, then $\Lambda$ is uniformly commensurated and  $q\circ\tau\circ \rho\circ \iota$ is a virtual uniform lattice embedding.
	\end{enumerate}
\end{prop}
\begin{proof}
Lemma \ref{lem:conn-comp-lie}  implies there is some compact normal subgroup $K_1\vartriangleleft G$ such that $(G/K_1)^\circ$ has no non-trivial compact normal subgroups. By Lemma  \ref{lem:lattice_lemma_quotient}, postcomposing $\rho$ with the quotient map $G\to G/K_1$ and replacing $G$ with $G/K_1$, we can thus assume $G^\circ$ has no non-trivial compact normal subgroup. In particular, Theorem \ref{thm:gleason-yamabe} implies that $G^\circ$ is a connected Lie group.

We choose an open subgroup $L\leq G$ and a compact subgroup $K\vartriangleleft L$ as in Proposition \ref{prop:opensbgrp_prod} and let $\pi_L:L\cong G^\circ \times K\to G^\circ$ be the quotient map. Let $\Gamma_L\coloneqq \rho^{-1}(L)$. Since $L$ is commensurated in $G$, Lemma \ref{lem:preimage_comm} ensures  $\Gamma_L\alnorm \Gamma$. As $\rho(\Gamma_L)=\rho(\Gamma)\cap L$ and $L$ is open, Lemma \ref{lem:lattice_open} implies $\rho(\Gamma_L)$ is a lattice in $L$. Let $\omega\coloneqq \pi_L\circ \rho|_{\Gamma_L}:\Gamma_L\to G^\circ$.  Lemma \ref{lem:lattice_lemma_quotient} implies $\im(\omega)=\pi_L(\rho(\Gamma_L))$ is a lattice in $G^\circ$. Since $\ker(\rho)$ is finite, $\rho(\Gamma_L)$ is discrete and $\ker(\pi_L)$ is compact, $\ker(\omega)$ is finite and so $\omega$ is a virtual lattice embedding.

By Proposition \ref{prop:nilrad}, $G^\circ$ either has non-trivial simply connected nilradical, or is a centre-free  semisimple Lie group. Assume first $G^\circ$ is a  centre-free  semisimple Lie group.  
Proposition \ref{prop:map_to_product} ensures the map $\tau:G\to \Aut(G^\circ)\times G/G^\circ$ is continuous and open with finite index image. By Lemma \ref{lem:fin_cen},  $\Aut(G^\circ)$ is a centre-free semisimple Lie group with  finitely many components, whilst Proposition \ref{prop:idcomp} ensures $G/G^\circ$ is totally disconnected.  
Since $\omega$ is a virtual lattice embedding,  Lemma \ref{lem:fin_cen} ensures the composition \[\Gamma_L\xrightarrow{\omega}G^\circ\xrightarrow{\Ad}\Aut(G^\circ)\] is also a virtual lattice embedding. 

We claim $\Ad\circ \omega$ coincides with \[ \Gamma\xrightarrow{\rho}G\xrightarrow{\tau}\Aut(G^\circ)\times G/G^\circ\xrightarrow{q} \Aut(G^\circ)\] on $\Gamma_L$, where $q$ is the projection. Indeed, for each $g\in \Gamma_L$, we have $\rho(g)\in L\cong G^\circ\times K$ so that $\rho(g)=s_gk_g$ for unique $s_g\in G^\circ$ and $k_g\in K$.  Thus $\pi_L(s_gk_g)=s_g=\omega(g)$. The definition of $\tau$ in Proposition \ref{prop:map_to_product} says that the map $G\xrightarrow{q\circ \tau} \Aut(G^\circ)$ is induced by the action of $G$ on $G^\circ$ by conjugation. Therefore, for all $h\in G^\circ$, we have \[q(\tau(\rho(g)))(h)=s_gk_ghk_g^{-1}s_g^{-1}=s_ghs_g^{-1}=\Ad(\omega(g))(h),\] since $k_g\in K$ commutes with $G^\circ$. Thus $\Ad\circ \omega$ and $q\circ \tau\circ \rho$ agree on $\Gamma_L$, and so $q\circ \tau\circ \rho|_{\Gamma_L}$ is a virtual lattice embedding.

Suppose in addition that $\rho$ is a uniform virtual lattice embedding. Then Lemmas \ref{lem:lattice_lemma_quotient} and \ref{lem:lattice_open} imply 
 $\omega$ is a uniform lattice embedding, hence so is $\Ad\circ \omega$. The preceding paragraph ensures $q\circ \tau\circ \rho|_{\Gamma_L}$ is a virtual uniform lattice embedding, so  Lemma \ref{lem:engulfing} may be applied to the map $q\circ \tau\circ \rho$  to deduce $\Gamma_L$ is uniformly commensurated in $\Gamma$.

We now assume  $G^\circ$  has non-trivial simply-connected nilradical. Choose the maximal $k$ such that the term  $C_k(N)$  of the lower central series is non-zero.  Then $C_k(N)$ is a non-trivial simply-connected  abelian Lie group, hence is isomorphic to $\bbR^n$ for some $n>0$. Since $N$ is a characteristic subgroup of $G^\circ$ and $C_k(N)$ is a characteristic subgroup of $N$, $C_k(N)$ is a characteristic subgroup of $G^\circ$. By Proposition \ref{prop:opensbgrp_prod},  $C_k(N)\times K$ is a commensurated subgroup of $G$. Thus Lemma \ref{lem:preimage_comm} says $\Lambda\coloneqq \rho^{-1}(C_k(N)\times K)$ is a commensurated subgroup of $\Gamma$. 
As $\omega(\Gamma_L)$ is a lattice in $G^\circ$, Lemma \ref{lem:mostow_lattice_series} implies \[\omega(\Gamma_L)\cap C_k(N)=\omega(\Lambda)\] is a uniform lattice in $C_k(N)\cong \bbR^n$, hence is isomorphic to $\bbZ^n$. Since $\ker(\omega)$ is finite, Lemma \ref{lem:finite_index_abelian} ensures $\Lambda$ has a finite index subgroup isomorphic to $\bbZ^n$; such a subgroup is also commensurated.
\end{proof}
\begin{proof}[Proof of Theorem \ref{thm:modelgeom}]
	Let $\Gamma$ be a finitely generated group that does not contain a non-trivial finite rank free abelian commensurated subgroup. Let $Z$ be a model geometry of $\Gamma$ and let  $\rho:\Gamma\to \Isom(Z)$ be the geometric action of $\Gamma$ on $Z$.  Let $G=\Isom(Z)$. 
	If $G^\circ$ is compact, then $G$ is compact-by-(totally disconnected), hence Lemma \ref{lem:dom} and Proposition \ref{prop:td_graph} ensure $Z$ is dominated by a locally finite vertex-transitive graph. 
	
	If $G^\circ$ is non-compact, Proposition \ref{prop:lattice_structure} ensures there is a copci homomorphism $\lambda:G\to S\times Q$ with $S$ and $Q$ as in the statement of Proposition \ref{prop:lattice_structure}. By Lemma \ref{lem:dom} and Propositions \ref{prop:manifold_gp_action} and \ref{prop:td_graph}, there exists a symmetric  space $X$, a locally finite vertex-transitive graph $Y$, and copci homomorphisms $\phi:S\to \Isom(X)$ and  $\psi:Q\to \Isom(Y)$. The composition \[\Isom(Z)\xrightarrow{\lambda}S\times Q\xrightarrow{(\phi,\psi)}\Isom(X)\times \Isom(Y)\to \Isom(X\times Y)\] is also copci, where the last map  is the product action. Thus $Z$ is dominated by $X\times Y$.	
\end{proof}

\begin{proof}[Proof of Theorem \ref{thm:lattice_prod}]
Let $\Gamma$ be a finitely generated group that does not contain a non-trivial finite rank free abelian commensurated subgroup and suppose $\rho:\Gamma\to G$ is a lattice embedding. If $G^\circ$ is compact, then the quotient map $\phi:G\to G/G^\circ$ is the required proper continuous map to a totally disconnected locally compact group. If $G^\circ$ is non-compact, then we can apply Proposition \ref{prop:lattice_structure} to get a copci map $\phi:G\to S\times D$ with finite index open image, where $S$ and $D$ satisfy the required properties.
\end{proof}

\section{Building lattice embeddings}\label{sec:building_group}
In Section \ref{sec:structure}, we proved one direction of Theorem \ref{thm:mainintro}, showing that if a finitely generated group $\Gamma$ has a model geometry not dominated by a locally finite graph, this implies the existence of a certain commensurated subgroup. In this section we  complete the proof of Theorem \ref{thm:mainintro}, showing how to embed a group $\Gamma$ containing a certain commensurated subgroup $\Lambda$ into a locally compact group. The technical result of this section,  Theorem \ref{thm:build_lattice_embedding} shows that given a Hecke pair with $(\Gamma,\Lambda)$ with engulfing group $E$, $\Gamma$ is a uniform lattice in a locally compact group containing $E$ as a closed normal subgroup.

Given a Hecke pair $(\Gamma,\Lambda)$, we consider the set $\Gamma/\Lambda$ of left $\Lambda$-cosets. There is an action $\phi:\Gamma\to \Sym(\Gamma/\Lambda)$  given by left multiplication.  The group $\Sym(\Gamma/\Lambda)$ can be equipped with the topology of pointwise convergence. The \emph{Schlichting completion of $(\Gamma,\Lambda)$} consists of the pair $(G,\lambda)$, where $G=\overline{\im(\phi)}$ and $\lambda:\Gamma\to G$ is the corestriction of $\phi$ to $G$. For ease of notation, we sometimes refer to $G$ as the Schlichting completion.  The following properties characterise Schlichting completions:
\begin{prop}[{\cite[Lemma 3.6]{shalomwillis13}}]
If $(G,\lambda)$ is the Schlichting completion of the   Hecke pair $(\Gamma,\Lambda)$, then:
	\begin{enumerate}
		\item $G$ is a totally disconnected locally compact group;
		\item $\lambda(\Gamma)$ is dense in $G$;
		\item There is a compact open subgroup $K\leq G$ such that $\rho^{-1}(K)=\Lambda$.
	\end{enumerate}
\end{prop}

We complete the proof of Theorem \ref{thm:mainintro}  by proving the following theorem.
\begin{thm}\label{thm:build_lattice_embedding}
	Let $\Gamma$ be a finitely generated group and let  $(\Gamma,\Lambda)$ be a Hecke pair with engulfing triple $(E,\rho,\Delta)$. Then there exists a locally compact group $G$ and a virtual uniform lattice embedding $\gamma:\Gamma\to G$ such that  $G$ fits into the short exact sequence \[1\rightarrow E\xrightarrow{p} G\xrightarrow{q} Q\rightarrow 1\] and the following hold:
	\begin{enumerate}
		\item $(Q,\lambda)$ is the Schlichting completion of $(\Gamma,\Lambda_0)$ for some finite index subgroup $\Lambda_0\leq \Lambda$
		\item  $p$ is a topological embedding and $q$ is a topological quotient map;
		\item   $\lambda=q\circ \gamma$;
		\item For all $g\in \Gamma$ and $v\in E$, $p(\Delta(g)(v))=\gamma(g)p(v)\gamma(g)^{-1}$.
	\end{enumerate}   
\end{thm}

Before proving Theorem \ref{thm:build_lattice_embedding}, we first explain how to deduce Theorem \ref{thm:mainintro} from it.
\begin{proof}[Proof of Theorem \ref{thm:mainintro}]
	Suppose $\Gamma$ has a model geometry that is not dominated by a locally finite graph. By Proposition \ref{prop:modgeom_vs_lattice}, $\Gamma$ is virtually a uniform lattice in a locally compact group $G$ with $G^\circ$ non-compact. Then Proposition \ref{prop:lattice_structure} says $\Gamma$ contains a commensurated subgroup $\Lambda$ as in the statement of Theorem \ref{thm:mainintro}.
	
	Conversely, suppose $\Gamma$ contains a commensurated subgroup $\Lambda$ as in the statement of Theorem \ref{thm:mainintro}, i.e.\ $\Lambda$ is either finite rank free abelian  or is uniformly commensurated and virtually isomorphic to a uniform lattice in a semisimple Lie group.  Propositions \ref{prop:engulf_abelian} and \ref{prop:engulf_commen_lattice} imply that in both cases, $(\Gamma,\Lambda)$ has a non-compact virtually connected  Lie engulfing group $E$. By Theorem \ref{thm:build_lattice_embedding}, $\Gamma$ is a virtual uniform lattice in a locally compact group $G$ containing a closed normal subgroup topologically isomorphic to  $E$. Since $E^\circ$ is non-compact, neither is  $G^\circ$. Therefore,  Proposition \ref{prop:modgeom_vs_lattice} implies there exists a  model geometry of $\Gamma$ that is not dominated by a locally finite vertex-transitive graph.
\end{proof}

We now begin our proof of Theorem \ref{thm:build_lattice_embedding}. Let $(\Gamma,\Lambda)$ and $(E,\rho,\Delta)$ be as in Theorem \ref{thm:build_lattice_embedding}. We automatically know from the Definition \ref{defn:engulf} that for all $g\in \Gamma$, the equation $\Delta(g)(\rho(h))=\rho(ghg^{-1})$ holds for all $h$ in \emph{some} finite index subgroup of $\Lambda$. However, to facilitate continuity arguments needed in the proof of Theorem \ref{thm:build_lattice_embedding}, we would like to show  that this subgroup can be taken to be  the intersection of finitely many conjugates of $\Lambda$.  The following lemma says this can  be arranged after replacing $\Lambda$ with a suitable finite index subgroup.
\begin{lem}\label{lem:intersection_engulf}
	Let $(\Gamma,\Lambda)$ be a Hecke pair with $\Gamma$ finitely generated and with engulfing triple $(E,\rho,\Delta)$.  Then there exists a finite index subgroup $\Lambda_0\leq \Lambda$ such that for every $g\in \Gamma$, there is a finite set $F_g\subseteq \Gamma$ such that $e,g^{-1}\in F_g$ and   $\Delta(g)(\rho(h))=\rho(ghg^{-1})$ for all $h\in \cap_{f\in F_g} f\Lambda_0f^{-1}$.
\end{lem}
\begin{proof}
	Let $S$ be a finite symmetric  generating set of $\Gamma$ containing the identity. Definition \ref{defn:engulf} ensures that for each $s\in S$, there is some finite index subgroup $\Lambda_s\leq \Lambda\cap s^{-1}\Lambda s$ such that $\Delta(s)(\rho(h))=\rho(shs^{-1})$ for all $h\in \Lambda_s$. Set $\Lambda_0\coloneqq \cap _{s\in S}\Lambda_s$. We claim that $\Lambda_0$ satisfies the required property. Indeed, if $g\in \Gamma$, then   $g=s_1s_2\dots s_n$  for some $s_1,\dots,s_n\in S$. Set $g_i=s_{i}\dots s_n$ for all $1\leq i\leq n$ and let $g_{n+1}=1$. Thus $g_1=g$ and $s_ig_{i+1}=g_i$ for all $1\leq i\leq n$.   Let  $F_g\coloneqq \{1,g_1^{-1}, \dots, g_n^{-1}\}$. If $h\in \cap_{f\in F_g} f\Lambda_0f^{-1}$, then for each $1\leq i\leq n$, $g_{i+1}hg_{i+1}^{-1}\in \Lambda_0\leq \Lambda_{s_i}$ so that \[\Delta(s_i)(\rho(g_{i+1}hg_{i+1}^{-1}))=\rho(s_ig_{i+1}hg_{i+1}^{-1}s_i^{-1})=\rho(g_{i}hg_{i}^{-1}).\] It follows that $\Delta(g)(\rho(h))=\rho(ghg^{-1})$ as required.
\end{proof}
We prove   Theorem \ref{thm:build_lattice_embedding} by defining a group extension $G$ using generalised cocycles, also called factor sets, as described in  \cite[\S IV.6]{brown1982cohomology}.
Suppose we are given groups $E$ and $Q$ and a homomorphism $\phi:Q\rightarrow \Out(E)$. We want to study group extensions \[1\to E\to G\to Q\to 1\] such that the action of $G$ on $E$ by  conjugation  descends to the given map $\phi$.  To do so, we require the following data:
\begin{itemize}
	\item A map $\omega:Q\rightarrow \Aut(E)$ such that $\omega(q)$ represents the outer automorphism $\phi(q)$ and $\omega(1)=1$.
	\item A map $f:Q\times Q\rightarrow E$ satisfying the following conditions.
	\begin{enumerate}
		\item  The \emph{generalised cocycle condition}
		\[f(g,h)f(gh,k)=\omega(g)(f(h,k))f(g,hk)\] for all $g,h,k\in Q$.
		\item The 
		\emph{compatibility condition}
		\[ \omega(g)\omega(h)=\Inn(f(g,h))\omega(gh)\]
		for all $g,h\in Q$, where $\Inn(a)\in \Aut(E)$ is the inner automorphism  $b\mapsto aba^{-1}$.
		\item The \emph{normalisation condition} \[f(1,g)=f(g,1)=1\] for all $g\in Q$.
		
	\end{enumerate}
\end{itemize}
We then define a group $G_{\omega,f}$ that is equal as a set to $E\times Q$ with multiplication defined by \begin{align}
	(a,g)\cdot (b,h)=(a\omega(g)(b)f(g,h),gh).\label{eqn:mult_extn}
\end{align}

The group $G_{\omega,f}$ is a well-defined extension of $E$ by $Q$ and the  conjugation action $G_{\omega,f}\rightarrow \Aut(E)$ descends to the given map $\phi:Q\rightarrow \Out(E)$. Moreover, if $E$ and $Q$ are locally compact topological groups and $\omega$ and $f$ are continuous, then $G_{\omega,f}$ is a locally compact topological group when endowed with the product topology.

\begin{rem}
	As noted by Hochschild \cite{hochschild1951group}, for a continuous short exact sequence \[1\rightarrow E\rightarrow G\rightarrow Q\rightarrow 1\] of locally compact topological groups, we cannot necessarily choose  a continuous generalised cocycle $f$ as above.
\end{rem}

Although the proof of Theorem \ref{thm:build_lattice_embedding} given below works for general $E$, we encourage the reader the keep in mind the special case where $\Lambda\cong \bbZ^n$ and $\Gamma\cong \bbR^n$, which encapsulates the difficulties and intricacies of the general case. In  the situation  where  $E$ has trivial centre, the proof can likely  be substantially simplified.
\begin{proof}[Proof of Theorem \ref{thm:build_lattice_embedding}]
Since the proof is rather long, we split it into several parts.

\textbf{Part 1: Defining the generalised cocycle.}	
	Let $(\Gamma,\Lambda)$ be a Hecke pair with an engulfing triple $(E,\rho,\Delta)$. Pick a finite index subgroup of $\Lambda$ such that the conclusion of Lemma \ref{lem:intersection_engulf} holds. Without loss of generality, we replace $\Lambda$ by such a subgroup. Thus for every $g\in \Gamma$, there is a finite set $F_g\subseteq \Gamma$ such that   $\Delta(g)(\rho(h))=\rho(ghg^{-1})$ for all $h\in \Lambda_g\coloneqq \cap_{f\in F_g} f\Lambda f^{-1}$.

	 Let $\sigma:\Gamma/\Lambda\rightarrow \Gamma$ be a section of the quotient map $g\mapsto g\Lambda$. We  assume $\sigma(\Lambda)=1$.  For each $g\in \Gamma$, let $\sigma_g=\sigma(g\Lambda)$  and $h_g=\sigma_g^{-1}g\in \Lambda$ so that $g=\sigma_gh_g$. We set $v_g\coloneqq \rho(h_g)\in E$ and define $\omega:\Gamma\rightarrow \Aut(E)$ by $\omega(g)=\Delta(\sigma_g)$. 
	 Although $\omega$ is not typically a homomorphism,  since $\Delta$ is a homomorphism and \[\omega(g)=\Delta(\sigma_g)=\Delta(gh_g^{-1})=\Delta(g)\Inn(v_g^{-1}),\] $\omega$ descends to a homomorphism $\Gamma\rightarrow \Out(E)$.  Since $\sigma(\Lambda)=1$, we see $v_1=1$ and  $\omega(1)=1$.
	 
	 Observe that for all $\beta\in \Aut(E)$ and $k\in E$, we have the identity $\beta \Inn(k)\beta^{-1}=\Inn(\beta(k))$.
	 Thus for all $g,k\in \Gamma$, we have \begin{align*}
		\omega(g)\omega(k)\omega(gk)^{-1}
		&=\Delta(g)\Inn(v_g^{-1})\Delta(k)\Inn(v_k^{-1})\Inn(v_{gk})\Delta(k)^{-1}\Delta(g)^{-1}\\
		& =\Delta(g)\Inn(v_g^{-1}\Delta(k)(v_k^{-1}v_{gk}))\Delta(g)^{-1}\\
		& =\Inn(\Delta(g)(v_g^{-1})\Delta(gk)(v_k^{-1}v_{gk})).
	\end{align*}
We thus set $f(g,k)=\Delta(g)(v_g^{-1})\Delta(gk)(v_k^{-1}v_{gk})$. By construction, the compatibility condition   \[\omega(g)\omega(k)=\Inn(f(g,k))\omega(gk)\] is satisfied for all $g,k\in \Gamma$. We now see that for all $g,h,k\in \Gamma$, we have 
\begin{align*}
	\omega(g)&(f(h,k))f(g,hk)=\Delta(g)\Inn(v_g^{-1})(\Delta(h)(v_h^{-1})\Delta(hk)(v_k^{-1}v_{hk}))f(g,hk)\\
	&=\Delta(g)(v_g^{-1}\Delta(h)(v_h^{-1})\Delta(hk)(v_k^{-1}v_{hk})v_g)\Delta(g)(v_g^{-1})\Delta(ghk)(v_{hk}^{-1}v_{ghk})\\
	&=\Delta(g)(v_g^{-1})\Delta(gh)(v_h^{-1})\Delta(ghk)(v_{k}^{-1}v_{ghk})\\
	&=\Delta(g)(v_g^{-1})\Delta(gh)(v_h^{-1}v_{gh})\Delta(gh)(v_{gh}^{-1})\Delta(ghk)(v_{k}^{-1}v_{ghk})\\
	&=f(g,h)f(gh,k)
\end{align*}
and so the  generalised cocycle condition holds. Moreover, since $v_1=1$, we see that $f(g,1)=f(1,g)=1$ for all $g\in \Gamma$; thus $f$ is normalised.

\textbf{Part 2: Showing $\omega$ and $f$ are continuous}	
We say a sequence $(g_i)$ in $\Gamma$ is a \emph{Cauchy sequence} if  for every $g\in \Gamma$, there exists a number $N$ such that $g_j^{-1}g_i\in g\Lambda g^{-1}$ for all $i,j\geq N$.
If $(g_i)$ is a Cauchy sequence in $\Gamma$, then  $g_j\Lambda=g_i\Lambda$ for sufficiently large $i,j$, and so \begin{align}
	\omega(g_i)=\Delta(\sigma(g_i\Lambda))=\Delta(\sigma(g_j\Lambda))=\omega(g_j)
\end{align}
 for  sufficiently large $i,j$.
 
Suppose $(g_i)$ and $(k_i)$ are Cauchy sequences. We pick $N_0$ large enough such that $k_i^{-1}k_j\in \Lambda$ for all $i,j\geq N_0$. Set $k\coloneqq k_{N_0}$. We pick $N\geq N_0$ large enough such that $g_i^{-1}g_j\in \Lambda\cap k\Lambda k^{-1}\cap \Lambda_{k^{-1}}$  for all $i,j\geq N$. Set $g=g_N$.
By the choice of $N$, $g$ and $k$, for each $i,j\geq N$ we have \[\text{$g_i=gh_{g,i}$, $k_i=kh_{k,i}$ and $g_ik_j=gk h_{gk,i,j}$}\]  for some  $h_{g,i},h_{k,i},h_{gk,i,j} \in \Lambda $.  Since $g_ik_j=gh_{g,i}kh_{k,j}=gkh_{gk,i,j}$, we deduce \[h_{k,j}h_{gk,i,j}^{-1}=k^{-1}h_{g,i}^{-1}k\]
for all $i,j\geq N$. 
For all $i\geq N$, we see that $\sigma_g=\sigma(g\Lambda)=\sigma(g_i\Lambda)=\sigma_{g_i}$. Similarly, $\sigma_k=\sigma_{k_i}$ and $\sigma_{gk}=\sigma_{g_ik_j}$ for all $i,j\geq N$. Therefore, \[\text{$h_{g_i}=h_gh_{g,i}$, $h_{k_i}=h_kh_{k,i}$ and $h_{g_ik_j}=h_{gk}h_{gk,i,j}$}\] for all $i,j\geq N$.

Now for $i\geq N$, we have $h_{g,i}^{-1}\in \Lambda_{k^{-1}}$ and so $\Delta(k^{-1})(\rho(h_{g,i}^{-1}))=\rho (k^{-1}h_{g,i}^{-1}k)$. Putting everything together,  we see that for all $i,j\geq N$ we have
\begin{align*}
	 f(g_i,k_j)&=\Delta(g_i)(v_{g_i}^{-1})\Delta(g_ik_j)(v_{k_j}^{-1}v_{g_ik_j})\\
	 & =\Delta(g_i)(\rho(h_{g_i}^{-1}))\Delta(g_ik_j)(\rho(h_{k_j}^{-1}h_{g_ik_j}))\\
	 &=\Delta(gh_{g,i})(\rho(h_{g,i}^{-1}h_g^{-1}))\Delta(gkh_{gk,i,j})(\rho(h_{k,j}^{-1}h_k^{-1}h_{gk}h_{gk,i,j}))\\
	&=\Delta(g)(\rho(h_g^{-1}h_{g,i}^{-1}))\Delta(gk)(\rho(h_{gk,i,j}h_{k,j}^{-1}h_k^{-1}h_{gk}))\\
	&=\Delta(g)(v_g^{-1}\rho(h_{g,i}^{-1}))\Delta(gk)(\rho(h_{gk,i,j}h_{k,j}^{-1})v_k^{-1}v_{gk})\\
	&=\Delta(g)(v_g^{-1})\Delta(gkk^{-1})(\rho(h_{g,i}^{-1}))\Delta(gk)(\rho(h_{gk,i,j}h_{k,j}^{-1})v_k^{-1}v_{gk})\\
	&=\Delta(g)(\rho(h_g^{-1}))\Delta(gk)(\rho(k^{-1}h_{g,i}^{-1}kh_{gk,i,j}h_{k,j}^{-1}))\Delta(gk)(v_k^{-1}v_{gk}))\\
	&=\Delta(g)(\rho(h_g^{-1}))\Delta(gk)(\rho(h_{k,j}h_{gk,i,j}^{-1}h_{gk,i,j}h_{k,j}^{-1}))\Delta(gk)(v_k^{-1}v_{gk}))\\
	 &=\Delta(g)(v_g^{-1})\Delta(gk)(v_k^{-1}v_{gk})\\
	&=f(g,k). 
\end{align*}

\textbf{Part 3: Defining the locally compact group.}	
Let $(Q,\lambda)$ be the Schlichting completion of $(\Gamma,\Lambda)$. It follows from the definitions that $(g_i)$ is a Cauchy sequence in $\Gamma$ if and only if $(\lambda(g_i))$ converges in $ Q$. Since $\lambda(\Gamma)$ is dense in $Q$ and $\omega$ and $f$ are eventually constant on any pair of Cauchy sequences,  it follows that $\omega$ and $f$ uniquely induce  continuous maps $\hat\omega:Q\rightarrow \Aut(E)$ and $\hat f: Q\times  Q\rightarrow E$ such that $\hat\omega(\lambda(g))=\omega(g)$ and $\hat f(\lambda(g),\lambda(k))=f(g,k)$ for all $g,k\in \Gamma$. Since $\omega$ and $f$ satisfy the compatibility condition, generalised cocycle condition, and are normalised,  $\hat \omega $ and $\hat f$ also satisfy these properties. Thus $\hat\omega$ and $\hat f$ define a locally compact topological group $G=G_{\hat\omega,\hat f}$ as above, i.e.\ $G$ is equal to $E\times Q$ as a set, equipped with the product topology and with multiplication defined by (\ref{eqn:mult_extn}).

We consider the homomorphisms $p:E\to G$ given by $v\mapsto (v,1)$ and $q:G\to Q$ given by $(v,g)\to g$. Since $p$ is injective, $q$ is surjective, and  $q\circ p$ is zero, we have a short exact sequence $1\to E\xrightarrow{p} G\xrightarrow{q} Q\to 1$. Moreover, since  $G$ is homeomorphic to $E\times Q$, $p$ is a topological embedding  and $q$ is a topological quotient map.

\textbf{Part 4: Defining the homomorphism to $G$.}	
We define $\gamma:\Gamma\rightarrow G$ by $\gamma(g)=(\Delta(g)(v_g),\lambda(g))$. We note that $q\circ \gamma=\lambda$ as required. We claim $\gamma$ is a virtual uniform lattice embedding, first showing $\gamma$ is a homomorphism. For all $g,k\in \Gamma$, we have \begin{align*}
	\gamma(g)\gamma(k)&= (\Delta(g)(v_g),\lambda(g))(\Delta(k)(v_k),\lambda(k))\\
	&= (\Delta(g)(v_g)\hat\omega(\lambda(g))(\Delta(k)(v_k))\hat f(\lambda(g),\lambda(k)),\lambda(g)\lambda(k))\\
	&= (\Delta(g)(v_g)\omega(g)(\Delta(k)(v_k))f(g,k),\lambda(gk))\\
	&= (\Delta(g)(v_g)\Delta(g)(v_g^{-1}\Delta(k)(v_k)v_g)\Delta(g)(v_g^{-1})\Delta(gk)(v_k^{-1}v_{gk}),\lambda(gk))\\
	&= (\Delta(g)(\Delta(k)(v_k))\Delta(gk)(v_k^{-1}v_{gk}),\lambda(gk))\\
	&= (\Delta(gk)(v_{gk}),\lambda(gk))=\gamma(gk)
\end{align*}
verifying that $\gamma$ is a homomorphism. 

\textbf{Part 5:  $\im(\gamma)$ is cocompact.}	
We now show $\gamma(\Gamma)$  is cocompact. 
Since $\rho:\Lambda\to E$ is a uniform lattice embedding, there is compact set $K\subseteq E$ such that $\rho(\Lambda)K=E$. Since $(Q,\lambda)$ is a Schlichting completion of $(\Gamma,\Lambda)$, there is a compact open subgroup $U\leq Q$ such that $\lambda^{-1}(U)=\Lambda$.  We claim $G=\gamma(\Gamma)(K, U)$, which will show $\gamma(\Gamma)$ is cocompact. 

Let $(g_1,g_2)\in G$. Since $\lambda(\Gamma)$ is dense in $Q$, there is some $g\in \Gamma$ such that $g_2\in \lambda(g)U$. As $\rho(\Lambda)K=E$, there is some $x\in K$ and $h\in \Lambda$ such that $\rho(h)x=\Delta(g^{-1})(g_1)\rho(h_g)^{-1}$. As $g_2\in \lambda(g)U=\lambda(gh)U$, we can choose $y\in U$ such that $g_2=\lambda(gh)y$.  We will show $\gamma(gh)(x,y)=(g_1,g_2)$. 

Since $\lambda(\Lambda)$ is dense in $U$, we can pick a sequence $(z_i)$ in $\Lambda$ such that $(\lambda(z_i))$ converges to $y$. Since $(z_i)$ is a Cauchy sequence,  the sequence  $\hat f(\lambda(gh),\lambda(z_i))$ is eventually constant and thus equal to $\hat f(\lambda(gh),y)$ for  $i$ sufficiently large. We can thus pick some $z\in \Lambda$ such that $ f(gh,z)=\hat f(\lambda(gh),y)$. Since $h,z\in \Lambda$, we have $h_z=z$ and $h_{ghz}=h_ghz$.   We thus deduce  \begin{align*}
	\gamma(gh)(x,y)&=(\Delta(gh)(v_{gh}),\lambda(gh))\cdot(x,y)\\
	&=(\Delta(gh)(v_{gh})\omega(gh)(x)f(gh,z),g_2)\\
	&=(\Delta(gh)(v_{gh})\Delta(gh)(v_{gh}^{-1}xv_{gh})\Delta(gh)(v_{gh}^{-1})\Delta(ghz)(v_z^{-1}v_{ghz}),g_2)\\
	&=(\Delta(gh)(x)\Delta(ghz)(v_{z}^{-1}v_{ghz}),g_2)\\
	&=(\Delta(g)(\rho(h)x\rho(h)^{-1})\Delta(g)(\rho(hzz^{-1}h_ghzz^{-1}h^{-1})),g_2)\\
	&=(\Delta(g)(\rho(h)x\rho(h_g)),g_2)\\
	&=(\Delta(g)(\Delta(g^{-1})(g_1)),g_2)\\
	&=(g_1,g_2).
\end{align*}
This shows $\gamma(\Gamma)(K,U)=G$ as required.

\textbf{Part 6:  $\gamma$ is proper.}	
We show $\gamma$ is proper by demonstrating that for every compact $V\subseteq G$, $\gamma^{-1}(V)$ is finite. Since $\gamma(\Gamma)$ is cocompact,  this will show $\gamma$ is a virtual uniform lattice embedding. As $G$ is identified as a set with $E\times Q$ and is equipped with the product topology, there is a continuous projection $\pi:G\to E$, which is typically not a homomorphism. Observe that for all $h\in \Lambda$, we have $\pi(\gamma(h))=\Delta(h)(v_h)=\rho(h)\rho(h)\rho(h)^{-1}=\rho(h)$, hence $\pi\circ \gamma|_{\Lambda}=\rho$.

Let $V\subseteq G$ be compact and let $U\leq Q$ be a compact open subgroup as above such that $\lambda^{-1}(U)=\Lambda$. Let $L=q^{-1}(U)$, where $q:G\to Q$ is the quotient map.  Then $L$ is an open subgroup of $G$, so $V$ is contained in finitely many left $L$ cosets. Since $\lambda(\Gamma)$ is dense, it intersects every coset of $U$, hence $\gamma(\Gamma)$ intersects every left coset of $L$. Thus there are finitely many $g_1,\dots,g_n\in \Gamma$ such that $V\subseteq \bigcup_{i=1}^n\gamma(g_i)L$. We thus deduce \[\gamma^{-1}(V)\subseteq \bigcup_{i=1}^n g_i \gamma^{-1}(L\cap \gamma(g_i)^{-1} V).\] Since $\gamma^{-1}(L)=\Lambda$ and $\pi\circ \gamma|_\Lambda=\rho$, 
we see that  
\[\rho(\gamma^{-1}(L\cap \gamma(g_i)^{-1} V))=\pi(L\cap \gamma(g_i)^{-1} V)\] for each $1\leq i\leq n$.  As $\pi(L\cap \gamma(g_i)^{-1} V)$ is compact and $\rho$ is a virtual lattice embedding, each \[\gamma^{-1}(L\cap \gamma(g_i)^{-1} V)\subseteq \rho^{-1}(\pi(L\cap \gamma(g_i)^{-1} V))\] is finite. Thus $\gamma^{-1}(V)$ is finite.

\textbf{Part 7: $\Delta$ acts by conjugation}
 All that remains is to show if  $v\in E$ and $g\in \Gamma$, we have $p(\Delta(g)(v))=\gamma(g)p(v)\gamma(g)^{-1}$.
Indeed, \begin{align*}
	\gamma(g)p(v)&=(\Delta(g)(v_g),\lambda(g))(v,1)\\
	&=(\Delta(g)(v_g)\omega(g)(v) f(g,1),\lambda(g))\\
	&=(\Delta(g)(v_g)\Delta(g)(v_g^{-1}vv_g),\lambda(g))\\
	&=(\Delta(g)(vv_g),\lambda(g)).
\end{align*}Since $\gamma(g)^{-1}=(\Delta(g)^{-1}(f(g,g^{-1})^{-1})v_g^{-1},\lambda(g)^{-1}),$ we see \begin{align*}
	\gamma(g)p(v)\gamma(g)^{-1}&=(\Delta(g)(vv_g),\lambda(g))(\Delta(g)^{-1}(f(g,g^{-1})^{-1})v_g^{-1},\lambda(g)^{-1})\\
	&=(\Delta(g)(vv_g)\omega(g)(\Delta(g)^{-1}(f(g,g^{-1})^{-1})v_g^{-1})f(g,g^{-1}),1)\\
	&=(\Delta(g)(vv_g)\Delta(g)(v_g^{-1}\Delta(g)^{-1}(f(g,g^{-1})^{-1}))f(g,g^{-1}),1)\\
	&=(\Delta(g)(v),1)=p(\Delta(g)(v))
\end{align*} as required.
\end{proof}

\section{Applications to low dimensional groups}\label{sec:lowdim}
In this section we prove Theorems \ref{thm:cd2} and \ref{thm:cd3} using tools from group cohomology. We refer the reader to Brown's book, especially Chapter VIII, for the necessary background  \cite{brown1982cohomology}.  

The following lemma is our starting point, which says uniform lattices in connected Lie groups possess desirable cohomological properties.
\begin{lem}[{\cite[VIII.9 Example 4 \& VIII.10 Example 1]{brown1982cohomology}}]\label{lem:poincare}
	If $\Gamma$ is a torsion-free uniform lattice in virtually connected Lie group $G$, then $\cd(\Gamma)=\dim(G/K)=\dim(G)-\dim(K)$, where $K$ is a maximal compact subgroup of $G$. Moreover, $\Gamma$ is a Poincar\'e duality group and so in particular, is of type $FP$. 
\end{lem}

We will prove Theorems \ref{thm:cd2} and \ref{thm:cd3} using the following theorem of the author, generalising work of Kropholler \cite{kropholler1990CD2,kropholler2006spectral}:
\begin{thm}[{\cite[Theorem 1.2 and Proposition 1.3]{margolis2019codim1}}]\label{thm:codim1}
	Let $\Gamma$ be a group containing a commensurated subgroup $\Lambda\alnorm \Gamma$ such that both $\Gamma$ and $\Lambda$ are of type $FP$.
	\begin{enumerate}
		\item If $\cd(\Gamma)=\cd(\Lambda)$, then $\Lambda$ is a finite index subgroup of $\Gamma$.
		\item If $\cd(\Gamma)=\cd(\Lambda)+1$, then $\Gamma$ splits as a finite graph of groups in which every vertex and edge group is commensurable to $\Lambda$.
	\end{enumerate}
\end{thm}
 In order to show the group in question is of type $FP$,  we appeal to the following result of Kropholler.
\begin{prop}[{\cite[Proposition 2.6]{kropholler2006spectral}}]\label{prop:typefp}
	Let $n\geq 1$ be a natural number. Let $\Gamma$ be a finitely generated group of cohomological dimension at most $n+1$. Let $\Lambda\alnorm \Gamma$ be a commensurated subgroup that is a Poincar\'e duality group of cohomological  dimension $n$. Then $\Gamma$ is of type $FP$.
\end{prop}

We also recall Thurston's definition of a model geometry:
\begin{defn}[{\cite[\S3]{thurston1997Threedimensional}}]\label{def:modgeom_thurs}
	A \emph{model geometry in the sense of Thurston}  is a pair $(G,X)$, where $X$ is a manifold and $G$ is a Lie group of diffeomorphisms of $X$ such that:
	\begin{enumerate}
		\item $X$  is simply-connected;
		\item $G$ acts transitively on $X$ with compact point stabilisers;
		\item $G$ is not contained in a larger group of diffeomorphisms of $X$ with compact point stabilisers;
		\item there is at least one compact manifold modelled on $(G,X)$.
	\end{enumerate}
\end{defn}
If $(G,X)$ is a model geometry in the sense of Thurston, we  equip $X$ with a $G$-invariant Riemannian metric. Although this metric is not unique, the isometry group of this Riemannian manifold is $G$ and hence is independent of the  metric. By a slight abuse of notation, we refer to $X$ equipped with such a  Riemannian metric  as a model geometry,  noting that two different metrics on $X$ may give rise to the same model geometry.

The following proposition is necessary to obtain the sharp conclusions of Theorem \ref{thm:cd2} and \ref{thm:cd3} without passing to finite index subgroups.
\begin{prop}\label{prop:extension_action}
	Let $\Gamma$ be a finitely generated group containing a finite index torsion-free subgroup $\Gamma'$ acting geometrically on a model geometry $X$ in the sense of Thurston. Then the $\Gamma'$ action on $X$ extends to a geometric action of $\Gamma$ on $X$.
\end{prop}
\begin{proof}
	Let $G=\Isom(X)$. Then $(G,X)$ is a model geometry in the sense of Definition \ref{def:modgeom_thurs}, and $X$ can be identified with the homogeneous space $G/K$ where $K$ is a maximal compact subgroup. As $\Gamma'$ is torsion-free and acts geometrically on $X$, $\Gamma'$  is a uniform lattice in $G$.  A result of Grunewald--Platonov says that $\Gamma$ is a uniform lattice in a Lie group $G_\Gamma$ with finitely many components \cite[Corollary 1.3]{grunewaldplatonov2004new}. Moreover, the group $G_\Gamma$ is a finite extension of $G$; see \cite[\S 3]{grunewaldplatonov2004new}. By Theorem \ref{thm:maxcompact}, there exists a maximal compact subgroup $K_\Gamma\leq G_\Gamma$ that meets every component of $G_\Gamma$. Thus  the inclusion $G\hookrightarrow G_\Gamma$ induces a diffeomorphism $G/K\to G_\Gamma/K_\Gamma$. This gives an  action  of  $G_\Gamma$ on $X$ by diffeomorphisms with compact point stabilisers that extends the action of $G$ on $X$. By Definition \ref{def:modgeom_thurs}, $G$ is a maximal such  group of diffeomorphisms of $X$, and so any $G$-invariant metric on $X$ is also $G_\Gamma$-invariant. Since $\Gamma$ is a uniform lattice in $G_\Gamma$, it follows the $\Gamma'$ action on $X$ extends to a geometric action of $\Gamma$  on $X$ as required.
\end{proof}
 Model geometries of dimensions two and three were listed by Thurston:
\begin{thm}[{\cite[Theorems 3.8.2 and 3.8.4]{thurston1997Threedimensional}}]\label{thm:thurs}
	Suppose $(G,X)$ is a model geometry in the sense of Thurston.
	\begin{itemize}
		\item If $\dim(X)=2$, then $X$ is either  $\bbE^2$, $\bbH^2$ or $S^2$.
		\item If $\dim(X)=3$, then $X$ is either $\bbE^3$, $\bbH^3$, $\bbH^2\times \bbE^1$, $\Nil$, $\widetilde{\SL(2,\bbR)}$, $\Sol$, $S^3$ or $S^2\times \bbE^1$.
	\end{itemize}
Moreover, $\bbH^2$ and $\bbH^3$ are the only symmetric spaces of non-compact type of dimension two or three.
\end{thm}
Combining Lemma \ref{lem:poincare} with Theorem \ref{thm:thurs},  we can explicitly describe the commensurated subgroups of cohomological dimension at most three that may occur in the statement of Theorem \ref{thm:mainintro}. This may also be deduced from the classification of symmetric spaces \cite{helgason1978differential}.
\begin{cor}\label{cor:lowdim_gps}
	Let $\Lambda$ be a torsion-free group that is either a finite rank free abelian group, or is a uniform lattice in a connected centre-free semisimple Lie group with no compact factors. Then $\Lambda$ is a Poincar\'e duality group. Moreover, the following hold.
	\begin{enumerate}
		\item If $\cd(\Lambda)=1$, then $\Lambda\cong \bbZ$.
		\item If $\cd(\Lambda)=2$, then either $\Lambda\cong \bbZ^2$ or $\Lambda$ acts geometrically on $\bbH^2$.
		\item If $\cd(\Lambda)=3$, then either $\Lambda\cong \bbZ^3$ or $\Lambda$ acts geometrically on $\bbH^3$.
	\end{enumerate}
\end{cor}

We recall  groups of finite cohomological dimension are  torsion-free, and that if $\Lambda\leq \Gamma$, then $\cd(\Lambda)\leq \cd(\Gamma)$. Moreover, any non-trivial group $\Lambda$ satisfies $\cd(\Lambda)>0$. These  facts will be used implicitly throughout the proofs of Theorems \ref{thm:cd2} and \ref{thm:cd3}. 

\begin{proof}[Proof of Theorem \ref{thm:cd2}]
	$\implies$: Suppose $\Gamma$ has a model geometry that is not dominated by a locally finite vertex-transitive graph. Then Theorem \ref{thm:mainintro} and Corollary \ref{cor:lowdim_gps} imply $\Gamma$ contains an infinite commensurated Poincar\'e duality subgroup $\Lambda$ such that  is either isomorphic to $\bbZ$ or $\bbZ^2$, or acts geometrically on $\bbH^2$.
	 Since $1\leq \cd(\Lambda)\leq \cd(\Gamma)=2$, Proposition \ref{prop:typefp} ensures $\Gamma$ is of type FP. 
	
	If $\Lambda\cong \bbZ^2$ or $\Lambda$ acts geometrically on $\bbH^2$, then $\cd(\Lambda)=\cd(\Gamma)$, so Theorem \ref{thm:codim1} implies $\Lambda$ is a finite index subgroup of $\Gamma$. As $\Lambda$ acts freely and cocompactly on either $\bbE^2$ or $\bbH^2$,  Proposition \ref{prop:extension_action} implies the torsion-free group $\Gamma$ acts freely and cocompactly on either $\bbE^2$ or $\bbH^2$, hence is a surface group and we are done.
	
	Now suppose $\Lambda\cong \bbZ$.  Theorem \ref{thm:codim1} ensures $\Gamma$ splits as a non-trivial graph of groups in which all vertex and edge groups are commensurable to $\Lambda$. Let $T$ be the associated Bass--Serre tree. 	If the tree $T$ is infinite-ended, then $\Gamma$ is a generalised Baumslag--Solitar group  of rank one and we are done. If $T$ is two-ended, then $\Gamma$ surjects onto $\bbZ$ or $D_\infty$ with kernel $K$ commensurable to $\Lambda$. As $\Gamma$ is torsion-free, $K\cong \bbZ$.  Thus  $\Gamma$ is virtually $\bbZ^2$ and so Proposition \ref{prop:extension_action} implies $\Gamma$ acts geometrically on $\bbE^2$ and hence is a surface group.
	
	$\impliedby$: If $\Gamma$  acts geometrically on $\bbE^2$ or $\bbH^2$, it has a model geometry not dominated by a locally finite graph.  If $\Gamma$  is a generalised Baumslag--Solitary group, it contains either a commensurated infinite cyclic subgroup, so Theorem \ref{thm:mainintro} implies $\Gamma$ has  a model geometry not dominated by a locally finite graph.
\end{proof}

The following lemma describes the structure of certain Schlichting completions, and  will arise in the proof of Theorem \ref{thm:cd3}.
\begin{lem}\label{lem:schlicht_tree}
	Let $\Gamma$ be a finitely generated group that splits as a finite graph of groups with all vertex and edge groups commensurable to $\Lambda$. Then $\Lambda$ is commensurated. Moreover, if $Q$ is the Schlichting completion of $\Gamma$ with respect to $\Lambda$, then $Q$ acts geometrically on a locally finite tree quasi-isometric to the Bass--Serre tree of $\Gamma$.
\end{lem}
\begin{proof}
	Let $T$ be the Bass-Serre tree associated to the given graph of groups $\cG$ of $\Gamma$, noting that the action of $\Gamma$ on $T$ is cocompact.  Since all vertex and edge groups of $\cG$ are commensurable to $\Lambda$, the tree $T$ is locally finite, hence every vertex and edge stabiliser of $T$ is commensurable to $\Lambda$. Thus $\Lambda$ is commensurated. Fix a finite symmetric generating set $S$ of $\Gamma$. Recall as in \cite[\S 3]{margolis2022discretisable}, the \emph{quotient space} $\Gamma/\Lambda$ is the set  $\Gamma/\Lambda$ of left $\Lambda$-cosets  equipped with the relative word metric. By \cite[Example 3.4 and Proposition 3.7]{margolis2022discretisable}, we see that as $\Lambda$ is commensurable to a vertex stabiliser of  $T$, the quotient space $\Gamma/\Lambda$ is quasi-isometric to $T$. 
	
	As there is a bijective correspondence between $\Lambda$-cosets and cosets of a compact open subgroup of $Q$ containing $\Lambda$ as a dense subgroup, the quotient space $\Gamma/\Lambda$ coincides precisely with the Cayley--Abels graph  of the Schlichting completion $Q$. Therefore, $Q$ is quasi-isometric to  $T$. A result of Cornulier, synthesising earlier work of Stallings, Abels and Mosher--Sageev--Whyte, shows $Q$ either acts geometrically on a locally finite tree, or admits a transitive geometric action on the real line \cite[Theorem 19.31]{cornulier2018quasi}. Since $Q$ is a Schlichting completion, hence totally disconnected, the latter cannot occur.
\end{proof}

\begin{proof}[Proof of Theorem \ref{thm:cd3}]
 $\implies$: Suppose $\Gamma$ has a model geometry that is not dominated by a locally finite vertex-transitive graph. Then Theorem \ref{thm:mainintro} implies $\Gamma$ contains an infinite commensurated subgroup $\Lambda$ that is one of the groups in the statement of Corollary \ref{cor:lowdim_gps}.  In the case where $\cd(\Lambda)=1$,  $\Lambda\cong \bbZ$ and we are done. We thus suppose $\cd(\Lambda)>1$. Then $2\leq\cd(\Lambda)\leq \cd(\Gamma)=3$.   Since $\Lambda$ is a Poincar\'e duality group, Proposition \ref{prop:typefp} ensures $\Lambda$ is of type FP. 
 
 Suppose $\cd(\Lambda)=3$.  Theorem \ref{thm:codim1} ensures  $\Lambda$ is a finite index subgroup of $\Gamma$. Moreover, Corollary \ref{cor:lowdim_gps} says either $\Lambda\cong \bbZ^3$ hence acts geometrically on $\bbE^3$, or  $\Lambda$ acts geometrically on $\bbH^3$.  Proposition \ref{prop:extension_action} implies $\Gamma$ acts geometrically on $\bbE^3$ or $\bbH^3$ as required.
 
 Now suppose  $\cd(\Lambda)=2$. By Corollary \ref{cor:lowdim_gps},  $\Lambda$ is either isomorphic to $\bbZ^2$ or acts geometrically on $\bbH^2$. Moreover, in the case $\Lambda$ acts geometrically on $\bbH^2$, it is uniformly commensurated.  Theorem \ref{thm:codim1} says $\Gamma$ splits as a non-trivial finite graph of groups in which all vertex and edge groups are commensurable to $\Lambda$.  Let $T$ be the associated Bass--Serre tree. The action of $\Gamma$ on $T$ is cocompact and without loss of generality, may assumed to be minimal.   There are four cases to consider, depending on whether  $\Lambda$ is isomorphic to $\bbZ^2$ or  acts geometrically on $\bbH^2$, and on whether $T$ is two-ended or infinite-ended.
 
In the case $\Lambda\cong \bbZ^2$ and $T$ is  infinite-ended,  $\Gamma$ is a generalised Baumslag--Solitar group of rank two and we are done. In the case $T$ is two-ended,  $\Gamma$ surjects onto $\bbZ$ or $D_\infty$ with kernel commensurable to $\Lambda$ hence virtually $\bbZ^2$.  It follows that  $\Gamma$ has a subnormal series with all factors infinite cyclic or finite. Thus $\Gamma$ is virtually polycyclic, hence some finite index subgroup  $\Gamma'\leq \Gamma$ is a lattice in a simply-connected solvable Lie group $G$ \cite[Theorem 4.28]{raghunathan1972discrete}. If $K\leq G$ is a maximal compact subgroup, the associated homogeneous space $G/K$ can be endowed with a left-invariant metric making it a three-dimensional model geometry. This model geometry must be one of $\bbE^3$, $\Nil$ or $\Sol$; see Theorem \ref{thm:thurs} and \cite[Figure 4.22]{thurston1997Threedimensional}. Since $\Gamma'$ acts geometrically on one of $\bbE^3$, $\Nil$ or $\Sol$,  Proposition \ref{prop:extension_action}  ensures $\Gamma$ does also. If $\Gamma$ acts geometrically on $\bbE^3$ or $\Sol$ we are done. If  $\Gamma$ acts geometrically on $\Nil$, then $\Gamma$ contains a normal infinite cyclic subgroup and we are also done.
 
 We now consider the case $\Lambda$ acts geometrically on $\bbH^2$. Since $\Lambda$ is uniformly commensurated,  Proposition \ref{prop:engulf_commen_lattice} ensures it has a semisimple Lie engulfing group, which may be assumed to be  $\Isom(\bbH^2)$. Theorem \ref{thm:build_lattice_embedding} implies there is a uniform lattice embedding  $\gamma:\Gamma\to G$, where  $G$ is a  locally compact group that fits into a short exact sequence \[1\to \Isom(\bbH^2)\to G\to Q\to 1,\] with $Q$  the Schlichting completion of $(\Gamma,\Lambda_0)$ for some finite index subgroup $\Lambda_0\leq \Lambda$. By observing that  $\Isom(\bbH^2)=\Aut(\Isom(\bbH^2)^\circ)$, we use Proposition \ref{prop:map_to_product} to deduce there is a copci homomorphism $G\to\Isom(\bbH^2)\times Q$. By Lemma \ref{lem:schlicht_tree}, there is a copci map $\psi:Q\to \Aut(T')$ for some locally finite tree $T'$ quasi-isometric to $T$. Thus there is a uniform lattice embedding $\Gamma\to \Isom(\bbH^2)\times \Aut(T')$, hence $\Gamma$ acts geometrically on $\bbH^2\times T'$. In the case $T'$ is infinite-ended, we are done. In the case $T'$ is two-ended, $\Gamma$ surjects onto $\bbZ$ or $D_\infty$ with kernel commensurable to $\Lambda$. Corollary \ref{cor:unif_comm} implies $\Gamma$ contains a two-ended normal subgroup, hence an infinite cyclic normal subgroup as required.

$\impliedby$: If $\Gamma$  acts geometrically on $\bbE^3$,  $\bbH^2$, $\Sol$  or $\bbH^2\times T$, it clearly has a model geometry not dominated by a locally finite graph. In the remaining cases,  $\Gamma$  contains a commensurated  subgroup isomorphic to $\bbZ$ or $\bbZ^2$, so Theorem \ref{thm:mainintro} implies $\Gamma$ has  a model geometry not dominated by a locally finite graph.
\end{proof}

\bibliography{bibliography/bibtex} 
\bibliographystyle{amsalpha}
\end{document}

%% file: torus_with_antennae_3.pdf_tex
%% Creator: Inkscape 1.2.1 (9c6d41e410, 2022-07-14), www.inkscape.org
%% Accompanies image file 'torus_with_antennae_3.pdf' (pdf, eps, ps)
\begingroup%
  \makeatletter%
  \providecommand\color[2][]{%
    \errmessage{(Inkscape) Color is used for the text in Inkscape, but the package 'color.sty' is not loaded}%
    \renewcommand\color[2][]{}%
  }%
  \providecommand\transparent[1]{%
    \errmessage{(Inkscape) Transparency is used (non-zero) for the text in Inkscape, but the package 'transparent.sty' is not loaded}%
    \renewcommand\transparent[1]{}%
  }%
  \providecommand\rotatebox[2]{#2}%
  \newcommand*\fsize{\dimexpr\f@size pt\relax}%
  \newcommand*\lineheight[1]{\fontsize{\fsize}{#1\fsize}\selectfont}%
  \ifx\svgwidth\undefined%
    \setlength{\unitlength}{407.45085949bp}%
    \ifx\svgscale\undefined%
      \relax%
    \else%
      \setlength{\unitlength}{\unitlength * \real{\svgscale}}%
    \fi%
  \else%
    \setlength{\unitlength}{\svgwidth}%
  \fi%
  \global\let\svgwidth\undefined%
  \global\let\svgscale\undefined%
  \makeatother%
  \begin{picture}(1,0.29515177)%
    \lineheight{1}%
    \setlength\tabcolsep{0pt}%
    \put(0,0){\includegraphics[width=\unitlength,page=1]{torus_with_antennae_3.pdf}}%
    \put(0.1546061,-0.02350823){\color[rgb]{0,0,0}\makebox(0,0)[lt]{\lineheight{1.25}\smash{\begin{tabular}[t]{l}$Z$\end{tabular}}}}%
    \put(0,0){\includegraphics[width=\unitlength,page=2]{torus_with_antennae_3.pdf}}%
    \put(0.58972793,-0.02315661){\color[rgb]{0,0,0}\makebox(0,0)[lt]{\lineheight{1.25}\smash{\begin{tabular}[t]{l}$\widetilde{Z}$\end{tabular}}}}%
  \end{picture}%
\endgroup%